\pdfoutput=1
\pdfoutput=1

\documentclass{article}

\usepackage{mathtools}

\usepackage{graphicx}
\usepackage{amsfonts}
\usepackage{amssymb}
\usepackage{amsmath}
\usepackage{amsthm}
\usepackage{url}
\usepackage{multirow}
\usepackage{enumerate}
\usepackage{hyperref}
\usepackage{epstopdf}
\usepackage{tikz}
\usepackage{pgfplots}
\usepackage{here}
\usepackage{bigstrut}
\usepackage{subcaption}
\usepackage{algorithm}
\usepackage{algpseudocode}
\usepackage[colorinlistoftodos]{todonotes}
\usepackage{dsfont}
\usepackage{adjustbox}
\usepackage{xcolor, colortbl}
\usepackage{afterpage}
\usepackage[indent]{parskip}
\usepackage[normalem]{ulem}

\DeclareMathOperator{\tr}{tr}

\DeclareMathOperator{\Diag}{Diag}
\DeclareMathOperator{\diag}{diag}

\DeclareMathOperator{\conv}{conv}
\DeclareMathOperator{\STAB}{STAB}
\DeclareMathAlphabet{\mymathbb}{U}{BOONDOX-ds}{m}{n} 

\newtheorem{thm}{Theorem}
\newtheorem{lem}[thm]{Lemma}
\newtheorem{cor}[thm]{Corollary}
\newtheorem{prop}[thm]{Proposition}

\newtheorem{obs}[thm]{Observation}

\newcommand{\R}{\mathbb{R}}

\begin{document}

\title{The exact subgraph hierarchy and its vertex-transitive variant for the stable set problem for Paley graphs}
\author{
	{Elisabeth Gaar}\thanks{University of Augsburg, Germany and Johannes Kepler University Linz, Austria, {\tt elisabeth.gaar@uni-a.de}}
	\and  {Dunja Pucher}\thanks{University of Klagenfurt, Austria,
		{\tt dunja.pucher@aau.at} }}
	
\date{}

\maketitle

\begin{abstract}
The stability number of a graph, defined as the cardinality of the largest set of pairwise non-adjacent vertices, is NP-hard to compute. The exact subgraph hierarchy (ESH) provides a sequence of increasingly tighter upper bounds on the stability number, starting with the Lovász theta function at the first level and including all exact subgraph constraints of subgraphs of order~$k$ into the semidefinite program to compute the Lovász theta function at level~$k$.  

In this paper, we investigate the ESH for Paley graphs, a class of strongly regular, vertex-transitive graphs. We show that for Paley graphs, the bounds obtained from the ESH remain the Lovász theta function up to a certain threshold level, i.e., the bounds of the ESH do not improve up to a certain level.

To overcome this limitation, we introduce the vertex-transitive ESH for the stable set problem for vertex-transitive graphs such as Paley graphs. We prove that this new hierarchy provides upper bounds on the stability number of vertex-transitive graphs that are at least as tight as those obtained from the ESH. Additionally, our computational experiments reveal that the vertex-transitive ESH produces superior bounds compared to the ESH for Paley graphs.


\end{abstract}


\section{Introduction}

This paper deals with the stable set problem, a fundamental combinatorial optimization problem. Given an undirected simple graph $G = (V, E)$ with $\vert V \vert = n$ vertices and $\vert E \vert = m$ edges, a subset of vertices $S \subseteq V$ is called a stable set if no two vertices of $S$ are adjacent. A stable set of largest possible cardinality in $G$ is called maximum stable set, and its cardinality is denoted by $\alpha(G)$. The stable set problem asks to determine $\alpha(G)$. It is an NP-hard problem contained in Karp's list~\cite{Karp72} from 1972. Therefore, unless P = NP, there is no polynomial time algorithm to solve the stable set problem.

A method that can be employed for solving the stable set problem is branch-and-bound, which strongly depends on good upper bounds. 
One possible upper bound on the stability number of a graph is given by the Lovász theta function~\cite{Lov:79}, which was introduced in 1979 and which can be computed as the optimal objective function value of a semidefinite program (SDP). Since its introduction, the Lovász theta function as well as its strengthenings have been extensively studied. One of the first refinements is the one proposed by Schrijver~\cite{Schrijver}. The Lovász theta function can also be tightened using a hierarchical approach. Such a hierarchy is a systematic procedure to strengthen a relaxation by adding additional variables or constraints, and, usually starting from the Lovász theta function at the first level of the hierarchy, it involves multiple levels, each providing a tighter relaxation than the previous one. The most prominent hierarchies based on semidefinite programming (SDP) are the ones from Sherali and Adams~\cite{SheraliAdamas}, Lovász and Schrijver~\cite{LovSch} and Lasserre~\cite{Lasserre}. Laurent investigated the computational application of these hierarchies to the stable set polytope in~\cite{Laurent}. 

In~\cite{AARW:15}, Adams, Anjos, Rendl, and Wiegele proposed the exact subgraph hierarchy for classes of NP-hard problems that are based on graphs and for which the projection of the problem onto a subgraph shares the same structure as the original problem. In particular, they considered the Max-Cut and the stable set problems. The approach from~\cite{AARW:15} for the latter starts by computing the Lovász theta function at the first level of the hierarchy. At level~$k$ of the hierarchy it considers all subsets of $k$ vertices and requires the exact subgraph constraint to be fulfilled, a constraint that ensures that for the corresponding subgraph the solution is exact. Computational studies of the exact subgraph hierarchy for the stable set, the Max-Cut, and graph coloring problems were performed by Gaar~\cite{Gaa:18} and by Gaar and Rendl~\cite{Gaa:20}.

The results for the stable set problem presented in~\cite{Gaa:18} reveal that for some graphs, higher levels of the hierarchy do not necessarily yield tighter bounds on their stability numbers than the ones obtained at the first level of the hierarchy. One such example is the Paley graph on $61$ vertices, for which an improvement of the bound was realized only on the sixth level of the hierarchy. On the other hand, for certain graphs, a significant improvement of the bounds on their stability numbers was achieved already on the second level of the hierarchy. A question that naturally arises and that we address in this paper is whether, for all Paley graphs, there is no improvement of the bound on the stability number up to a certain level of the exact subgraph hierarchy and, if so, whether there is a theoretical explanation for it.

So far, the stability numbers of Paley graphs have gained some attention in the literature. Several authors investigated closed-form bounds on stability numbers of Paley graphs already a long time ago, for example, Broere, Döman, and Ridley~\cite {Broere} and Cohen~\cite{Cohen}.
The best closed-form upper bound for a certain class of Paley graphs was given only recently by Hanson and Petridis~\cite{Hanson}. Upper SDP bounds on the stability numbers of Paley graphs based on block-diagonal SDP hierarchies for 0/1 programming were investigated by Gvozdenovic, Laurent, and Vallentin~\cite{Gvozdenovic2009}. Furthermore, Magsino, Mixon, and Parshall~\cite{Magsino} argued that one can obtain bounds on the stability number of Paley graphs by considering certain subgraphs only. In particular, they showed that the Schrijver refinement of the Lovász theta function applied on these subgraphs yields bounds that are sometimes better than the ones proposed in~\cite{Hanson}.
In this paper, we generalize the approach of~\cite{Magsino} for all vertex-transitive graphs, which eventually leads us to the introduction of a new hierarchy for their stable set problem.

\textbf{Contribution and outline } 
In this paper, we consider upper bounds on the stability numbers of Paley graphs $P_q$ of order $q$ based on exact subgraph constraints from two different aspects. First, we systematically investigate various levels of the exact subgraph hierarchy for Paley graphs, starting with the construction of an optimal solution for the SDP to compute the first level---the Lovász theta function. Then, we examine whether the upper bounds for the stability numbers of Paley graphs on higher levels of the exact subgraph hierarchy improve. In particular, we identify a specific level~$\ell(q) = \left\lfloor\frac{\sqrt{q}+3}{2}\right\rfloor$ for which adding exact subgraph constraints on subgraphs of orders ${2, \ldots, \ell(q)}$ fails to provide better bounds on the stability numbers for the Paley graph $P_q$. Moreover, for certain Paley graphs, we prove that adding exact subgraph constraints on subgraphs of order~$\ell(q) + 1$ also fails to improve the Lovász theta function as an upper bound on the stability number. As a consequence, the exact subgraph hierarchy does not perform well at providing upper bounds on the stability number of a Paley graph $P_q$, as it only has the potential to improve the bound starting from level~$\ell(q)+1$ or even $\ell(q) + 2$.

Second, to address this limitation, we present an alternative approach for the computation of upper bounds on the stability numbers of Paley graphs. For this purpose, we introduce the vertex-transitive exact subgraph hierarchy for the stable set problem for vertex-transitive graphs, which includes Paley graphs. To do so, we fix one vertex of the graph $G$ to be in a maximum stable set and apply the exact subgraph hierarchy to the local graph of $G$, i.e., to the subgraph of $G$ in which the remaining vertices of the maximum stable set can be. We prove that for any vertex-transitive graph, the bounds obtained from the vertex-transitive exact subgraph hierarchy are at least as good as those from the exact subgraph hierarchy. In our computational study, we demonstrate that this new hierarchy yields bounds on the stability numbers of Paley graphs that are significantly better than the bounds obtained from the exact subgraph hierarchy and that are even tighter than the ones proposed in~\cite{Hanson} and~\cite{Magsino}. Specifically, our computations reveal that for every Paley graph considered, the upper bound on the stability number at the second level of the vertex-transitive exact subgraph hierarchy is significantly tighter than at the first (and consequently $\ell(q)$-th) level of the exact subgraph hierarchy. This demonstrates that the vertex-transitive exact subgraph hierarchy effectively overcomes the issue of stagnant upper bounds for multiple levels of the exact subgraph hierarchy.

The rest of this paper is organized as follows. 
In Section~\ref{Prerequisites}, we give an overview of the theoretical concepts used in this work, including Paley graphs, the Lovász theta function, and the exact subgraph hierarchy. In Section~\ref{section_ESH_Paley}, we investigate various levels of the exact subgraph hierarchy for Paley graphs that do not yield improved upper bounds on their stability numbers. In Section~\ref{section_ESH_local}, we introduce the vertex-transitive exact subgraph hierarchy for the stable set problem for vertex-transitive graphs and investigate its effectiveness in finding upper bounds on stability numbers of Paley graphs. We conclude with a short discussion of our results and open questions in Section~\ref{Conclusion}. 

\textbf{Notation and terminology } We close this section with some notation and terminology. 
We denote by $\mathbb{N}_0$ the natural numbers including zero. The vector of all-ones of length $n$ is denoted by $\mathds{1}_n$ and we write $\mathds{1}_{n \times n} =\mathds{1}_n \mathds{1}^T_n$ for the matrix of all-ones. Let $0_n$ be the zero matrix and $I_n$ the identity matrix of order~$n$. We denote by $e_i$ the column $i$ of the matrix $I_n$. Furthermore, we set
\begin{align*}
	E_i &= e_i e_i^T \\
	E_{ij} &= (e_i + e_j)(e_i + e_j)^T. 
\end{align*}
If $a \in \mathbb{R}^n$, then $\Diag(a)$ denotes a $n \times n$ diagonal matrix with $a$ on the main diagonal, while if $X \in \mathbb{R}^{n \times n}$, then $\diag(X)$ is a vector of length $n$ containing the main diagonal of $X$.

For basic graph theoretical concepts, we use the notion of graphs, the order, the complement graph $\overline{G}$ of a graph $G$, induced subgraphs, graph isomorphisms, as well as neighbors and the set of neighbors $N(v)$ of a vertex~$v$ of Diestel~\cite{Diestel}. We recall that a clique is a subset of vertices such that every two distinct vertices in the clique are adjacent. The cardinality of a maximum clique in $G$ is denoted by $\omega(G)$. Since a stable set in $G$ is a clique in $\overline{G}$, we have $\alpha(G) = \omega(\overline{G})$.
A matrix $X \in \mathbb{R}^{n \times n}$ is circulant if $X_{i+1, j+1} = X_{ij}$ for all $i, j \in \mathbb{Z}_n = \{0, 1, \ldots, n - 1\}$.
A graph $G$ on~$n$ vertices is circulant if there exists a labeling of its vertices with $\mathbb{Z}_n$ such that its adjacency matrix is circulant. If $G$ is circulant, its complement $\overline{G}$ is also circulant.

Furthermore, a graph $G$ is vertex-transitive if its automorphism group acts transitively on $V$. Analogously, a graph $G$ is edge-transitive if its automorphism group acts transitively on $E$. A regular graph $G$ that is neither complete nor empty is said to be strongly regular with parameters $(n, r, a, c)$ if it is $r$-regular, if every pair of adjacent vertices has $a$ common neighbors, and if every pair of distinct non-adjacent vertices has $c$ common neighbors. See Godsil and Royle~\cite{Godsil} for more information on these definitions.

\section{Theoretical background}\label{Prerequisites}

In this section, we provide the theoretical background required for this paper. 

\subsection{Paley graphs}\label{Paley_graphs}

Paley graphs are named after the English mathematician Raymond Edward Alan Christopher Paley (1907-1933) and are closely related to the Paley construction for constructing Hadamard matrices, see~\cite{Paley}. Paley graphs were independently introduced by Sachs in~\cite{Sachs} and by Erd{\"o}s and R{\'e}nyi in~\cite{Erdos}. Some aspects of the history of the Paley graphs can be found in~\cite{Jones}.

Paley graphs are defined in the following way. Let $p$ be a prime, $s$ a positive integer, and let $q$ be a prime power, i.e.\ $q = p^s$, such that $q \equiv 1 \pmod{4}$. Then the Paley graph has as vertex set the elements of the finite field $\mathbb{F}_q = \{0, \ldots, q - 1\}$, with two vertices being adjacent if and only if their difference is a nonzero square in $\mathbb{F}_q$. The congruence condition on $q$ implies that $-1$ is a square in $\mathbb{F}_q$; hence, the graph is undirected. For given $q$, we denote such a Paley graph by $P_q$ and write $V_q = \{0, \ldots, q - 1\}$ and $E_q$ for the set of vertices and edges, respectively.

Paley graphs have been extensively studied due to their various properties; see, for instance,~\cite{Brouwer}. Among others, Paley graphs are quasi-random graphs, as shown in~\cite{Bollobas2001}, and their clique numbers play an important role in Ramsey theory; see, for instance,~\cite{Xu}. In the following text, we focus on the particular characteristics of Paley graphs necessary for determining and computing their stability numbers.

The first important property is that Paley graphs are self-complementary, as shown in~\cite{Bollobas2001}. This leads to the following observation.
\begin{obs}\label{alpha_equals_clique}
Let $P_q$ be a Paley graph. Then $\alpha(P_q) = \omega(P_q)$.
\end{obs}

Furthermore, Paley graphs are vertex and edge transitive, and they are strongly regular graphs with parameters $(q, \frac{q-1}{2}, \frac{q-5}{4}, \frac{q-1}{4})$, as also demonstrated in \cite{Bollobas2001}. A property of strongly regular graphs is that the adjacency matrix has only three different eigenvalues, as shown in~\cite{Godsil}. In particular, if~$G$ is a strongly regular graph with parameters $(n, r, a, c)$, then the eigenvalues of the matrix $A$ are $r$, as well as $\frac{1}{2}\Bigl((a-c) \pm \sqrt{(a-c)^2 + 4(r-c)}\Bigr)$. In the context of Paley graphs, this leads to the following observation.

\begin{obs}\label{eigenvalues_Paley}
The adjacency matrix $A$ of a Paley graph $P_q$ has eigenvalues $\frac{q-1}{2}$ and $\frac{-1 \pm \sqrt{q}}{2}$.
\end{obs}

Additionally, since Paley graphs are self-complementary, we note the following.

\begin{obs}\label{eigenvalue_Paley_complement}
The adjacency matrix $\overline{A}$ of $\overline{P}_q$ has the same eigenvalues as the adjacency matrix $A$ of $P_q$.
\end{obs}

Finally, we note that if $q$ is a prime, then $P_q$ is circulant and the subgraph of $P_q$ induced by the set of neighbors of the vertex $0$ is also circulant. For further details, see~\cite {Magsino}. 

Next, we investigate bounds on the stability numbers of Paley graphs from the literature. In~\cite{Cohen}, Cohen investigated lower bounds on the stability numbers of Paley graphs and showed that for any $P_q$ with $q = p^s$ for some prime $p$ and some positive integer $s$,
\begin{align*}
\alpha(P_q) \geq \frac{p}{p-1} \left( \frac{\frac{1}{2} \log q - 2 \log \log q}{\log 2} + 1 \right)
\end{align*}
holds. 

Regarding upper bounds on the stability numbers of Paley graphs, 
due to their regularity, one can apply Hoffman's bound~\cite{Hoffman} to derive a closed-form upper bound. In general, let $G$ be an
$r$-regular graph on $n$ vertices, and let $\tau$ denote the smallest eigenvalue of the adjacency matrix of~$G$. Then, according to Hoffman's bound,
\begin{align*}
\alpha(G) \leq \frac{n}{1 - \frac{r}{\tau}}. 
\end{align*}

In case of a Paley graph $P_q$, we have that $r = \frac{q-1}{2}$, and $\tau = \frac{-1 - \sqrt{q}}{2}$, according to Observation~\ref{eigenvalues_Paley}. Hence,
\begin{align}\label{Hoffman}
	\alpha(P_q) \leq \sqrt{q}
\end{align}
for every $P_q$. If $q$ is a square, then the upper bound~\eqref{Hoffman} is tight, as shown in~\cite{Broere}.

The upper bound for cases when $q$ is not a square and $q \neq 5$ was improved by Maistrelli and Penman~\cite{Maistrelli} to $\alpha(G_q) \leq \sqrt{q - 4}$. Recently, Hanson and Petridis~\cite{Hanson} considered Paley graphs $P_q$ where $q$ is a prime and gave the even better closed-form bound
\begin{align*}
	b_{H}(P_q) = \frac{\sqrt{2q-1}+1}{2},
\end{align*}
where equality holds for $q \in \{5, 13, 41\}$.

\subsection{Lovász theta function}\label{semidefinite_relaxation}

In 1979, Lovász~\cite{Lov:79} introduced the graph parameter Lovász theta function $\vartheta(G)$, which yields an upper bound on the stability number of $G$ and which can be computed as the optimal objective function value of an SDP in polynomial time to arbitrary fixed precision, as shown in~\cite{VanBoyd}. There are several different formulations for the Lovász theta function. A study of different variants of the Lovász theta function can be found in \cite{GALLI2017159}. In~\cite{LovSch}, Lovász and Schrijver showed that for a graph~$G$ on $n$ vertices the formulations
\begin{align}
	\vartheta(G) &= \max \Biggl\{\mathds{1}^T_n x \colon \begin{pmatrix} X & x \\ x^T & 1
	\end{pmatrix} \succeq 0, \diag(X) = x, X_{ij} = 0 ~~\forall \{i,j\} \in E\Biggr\} \label{theta_1} \\
	&= \max \{\langle \mathds{1}_{n \times n}, X \rangle \colon X \succeq 0, ~ \tr(X) = 1, ~ X_{ij} = 0 ~~\forall \{i,j\} \in E\} \label{theta_2}
\end{align}
 are equivalent. 
 
 The relationship between the formulations~\eqref{theta_1} and~\eqref{theta_2} has been studied by several authors. In~\cite[Lemma~2.8 and Theorem~2.9]{Gru:03}, Gruber and Rendl showed the following statements.

\begin{lem}~\label{opt_1}
	Let $(x^*, X^*)$ be a feasible solution of~\eqref{theta_1} with $\tr(X^*) > 0$. Then $X^\prime = \frac{1}{\tr(X^*)}{X^*}$ is a feasible solution of~\eqref{theta_2}.
\end{lem}

\begin{lem}\label{opt_2}
	Let $X^\prime$ be an optimal solution of~\eqref{theta_2}, and let $Y = \langle \mathds{1}_{n \times n}, X^\prime \rangle X^\prime$. Then $(\diag(Y), Y)$ is an optimal solution of~\eqref{theta_1}.
\end{lem}

Another property of the formulation~\eqref{theta_2} studied in the original work by Lovász~\cite[Corollary~2]{Lov:79} is the relationship between the Lovász theta function of a graph $G$ and its complement $\overline{G}$. In this context, Lovász proved the following.

\begin{lem}\label{theta_complement}
For any graph $G$ on $n$ vertices, 
$\vartheta(G) \vartheta(\overline{G}) \geq n$ holds.
\end{lem}

Additionally, if $G$ is vertex-transitive, the following theorem, also proven in~\cite[Theorem~8]{Lov:79}, holds.

\begin{thm}\label{theta_vertex_transitive}
If $G$ is a vertex-transitive graph on $n$ vertices, then
$\vartheta(G) \vartheta(\overline{G}) = n$ holds.
\end{thm}

\subsection{Exact subgraph hierarchy}\label{method_ESH}

We now recall the definitions of exact subgraph constraints 
and 
the exact subgraph hierarchy, whose aim it is to strengthen the Lovász theta 
function~$\vartheta(G)$ into the direction of $\alpha(G)$. We follow the 
presentation of 
Gaar~\cite{GaarVersionsESH}. 
For a graph $G = (V,E)$ with  $|V|=n$ and $V = \{0, \dots, n-1\}$
the squared stable set polytope 
$\STAB^{2}(G)$ is 
defined as 
\begin{align*}
\STAB^{2}(G) &= \conv\left\{ss^{T}: s \in 
\{0,1\}^{n}, \, s_{i}s_{j} = 0 ~ \forall 
\{i,j\} \in E \right\}, 
\end{align*} 
i.e., $\STAB^{2}(G)$ is the convex hull of the outer product of all 
incidence vectors of stable sets in $G$. It turns out that when the constraint 
$X \in \STAB^{2}(G)$ is added to the SDP $\eqref{theta_1}$ to compute the 
Lovász theta 
function $\vartheta(G)$, 
the optimal objective function 
is $\alpha(G)$, see 
Gaar~\cite{GaarVersionsESH}. 
Unfortunately, the resulting SDP is in general too large for 
computations, so the idea is to include the constraint $X \in \STAB^{2}(G)$ 
only partially for smaller subgraphs.  

In particular, for a subset of the vertices $I \subseteq V$, the subgraph of 
$G$ that is 
induced by~$I$ is denoted by $G_I$. 
Moreover, the submatrix of 
$X\in\R^{n\times n}$ which is indexed by~$I$ is denoted by 
$X_I=(X_{i,j})_{i,j\in I}$.   
Then the exact subgraph constraint (ESC) for $G_I$ is defined as
$
X_I \in \STAB^{2}(G_I).
$
Furthermore, 
let $J$ be a set of subsets of $V$. Then $z_{J}(G)$ is the optimal 
objective 
function value
of~\eqref{theta_1} with the ESC for 
every 
subgraph induced by a set in $J$, so
\begin{align*}
z_{J}(G) = & 
\max 
\Biggl\{\mathds{1}_n^T x : \begin{pmatrix} X & x \\ x^T & 1
\end{pmatrix} \succeq 0, \, \diag(X) = x, \, X_{ij} = 0 ~\forall \{i,j\} \in 
E, \\
&\qquad \qquad X_I \in \STAB^{2}(G_I) ~ \forall I \in J\Biggr\}.
\end{align*}
Finally, the $k$-th 
level of the exact subgraph hierarchy (ESH) is defined as 
$z_{k}(G) = z_{J_k}(G)$ for $J_k = \{I \subseteq V: \vert I \vert = k\}$, i.e., it 
contains the ESC for every subgraph of $G$ of order~$k$. For $k > n$, we define $z_k(G) = z_n(G)$ as a logical continuation of the ESH.

The ESH goes back to Adams, 
Anjos, 
Rendl and Wiegele~\cite{AARW:15}, who called the ESCs for subgraphs of order~$k$ the $k$-projection constraints, and who considered the hierarchy not only 
for the stable set problem, but also for the Max-Cut problem.
Furthermore,
\begin{align}
\label{ineq:relationship_z}
\alpha(G) = z_n(G) \leq \cdots \leq z_{k+1}(G) \leq z_k(G) \leq \cdots \leq z_1(G) = z_0(G) = \vartheta(G)
\end{align}
holds, so the higher the considered level~$k$ of the ESH is, 
the better the bounds on the stability number are.
Later on, Gaar~\cite{Gaa:18} and Gaar and Rendl~\cite{Gaa:20} 
considered the option of not including all ESCs of subgraphs of a certain 
order into \eqref{theta_1}, but to 
include the ESC only for some cleverly chosen subgraphs. 
They provided a framework to obtain strengthened upper bounds on the 
stability number by iteratively computing $z_{J}(G)$ by solving the corresponding SDP, searching for violated 
ESCs and updating $J$.
The huge advantage of the ESH over other SDP-based hierarchies that start from 
the Lovász theta function $\vartheta(G)$ is that the size of the square matrix that needs to be 
positive semidefinite remains $n+1$ on each level of the ESH, see 
Gaar~\cite{GaarVersionsESH} for more details.

Furthermore, the bounds obtained with the help of ESCs were utilized within a 
branch-and-bound algorithm to solve the stable set problem to optimality by 
Gaar, Siebenhofer and Wiegele~\cite{GaarSiebenhoferWiegeleStabBaB}. 
Moreover, Gaar~\cite{GaarVersionsESH} investigated including the ESCs not 
starting from the SDP $\eqref{theta_1}$ to compute the 
Lovász theta 
function $\vartheta(G)$, but beginning from the alternative SDP 
$\eqref{theta_2}$.

\section{The exact subgraph hierarchy for Paley graphs}\label{section_ESH_Paley}

In this section, we explore the bounds on the stability numbers of Paley graphs obtained using the ESH.

\subsection{Optimal solution for SDP to compute the Lovász theta function}\label{optimal_solution_for_Paley}

We first demonstrate that for Paley graphs, an optimal solution for the SDP to compute the Lovász theta function can be explicitly constructed. 
We start by recalling that Paley graphs are vertex-transitive. Thus, from Theorem~\ref{theta_vertex_transitive} we have that $\vartheta(P_q) \vartheta(\overline{P_q}) = q$. Additionally, Paley graphs are self-complementary. This leads to the following important observation about the Lovász theta function for Paley graphs.

\begin{obs}\label{theta_paley_graphs}
	Let $P_q$ be a Paley graph. Then $\vartheta(P_q) = \sqrt{q}$.
\end{obs}

Hence, for every $q$, we know the value of $\vartheta(P_q)$. 
As previously mentioned, in~\cite{Broere} it was proved that, if $q$ is a square, then $\alpha(P_q) = \sqrt{q}$. Thus, in this case $\vartheta(P_q)$ and~$\alpha(P_q)$ coincide. As a result, the Lovász theta function does not need to be improved as bound on $\alpha(P_q)$ because the following holds.
\begin{obs}\label{theta_exact_q_square}
    Let $P_q$ be a Paley graph such that $q$ is a square. Then
    \begin{align*}
    \alpha(P_q) = \sqrt{q} = \vartheta(P_q) = z_{k}(P_q)
    \end{align*}
    holds for all $k \leq n$.
\end{obs}

For all values of $q$, 
by combining Observation~\ref{theta_paley_graphs} with Lemmas~\ref{opt_1} and \ref{opt_2}, we are able to construct an optimal solution for~\eqref{theta_1} as follows. 

\begin{lem}\label{optimal_solution}
	Let $P_q$ be a Paley graph. Let $x^* \in \mathbb{R}^q$ and $X^* \in \mathbb{R}^{q \times q}$ with
	\begin{subequations}\label{solution}
		\begin{alignat}{3}
			x^*_i &= \frac{1}{\sqrt{q}} &&\quad \forall i \in V_q \label{sol_1}\\
			X^*_{ii} & = x^*_{i} &&\quad \forall i \in V_q \label{sol_2}\\
			X^*_{ij} &= 0 &&\quad \forall \{i,j\} \in E_q \label{sol_3} \\
			X^*_{ij} &= \frac{2}{q + \sqrt{q}} &&\quad \forall \{i,j\} \notin E_q \label{sol_4}.
		\end{alignat}
	\end{subequations}
	Then $(x^*, X^*)$ is an optimal solution of~\eqref{theta_1}.
\end{lem}

\begin{proof}
	First, we set $X^\prime = \frac{1}{\tr(X^*)}X^* = \frac{1}{\sqrt{q}}X^*$, and show that $X^\prime$ is a feasible solution for~\eqref{theta_2}. 
	
	From~\eqref{solution} we have that the constraints $\tr(X^\prime) = 1$ as well as $X^\prime_{ij} = 0$ for all edges~$\{i,j\}$ are satisfied. What is left to show is that $X^\prime \succeq 0$. Since $\frac{1}{\sqrt{q}} > 0$, it is sufficient to show that $X^* \succeq 0$. Let $D = \Diag(x^*)$. Then, we can write the matrix~$X^*$ as
	\begin{align*}
	X^* = D + \frac{2}{q + \sqrt{q}} \overline{A},
	\end{align*}
	where $\overline{A}$ is the adjacency matrix of the complement of $P_q$. From Observations~\ref{eigenvalues_Paley} and~\ref{eigenvalue_Paley_complement} we know that $\overline{A}$ has eigenvalues $\frac{q-1}{2}$ and $\frac{-1 \pm \sqrt{q}}{2}$. Since $D$ is a diagonal matrix with elements equal to $x^*_i$, the only eigenvalue of $D$ is $\frac{1}{\sqrt{q}}$. Therefore, the eigenvalues of the matrix $X^*$ are $1$, $\frac{2}{1 + \sqrt{q}}$ and $0$. Hence, the matrix $X^*$ is positive semidefinite, and thus $X^\prime$ is feasible for~\eqref{theta_2}. 
 
    Next, we note that every vertex $i \in P_q$ has $\frac{q - 1}{2}$ neighbors as well as $\frac{q - 1}{2}$ non-neighbors. Therefore, the sum of the $i$-th row of the matrix $X^*$ is
    \begin{align*}
    \frac{1}{\sqrt{q}} + \frac{q - 1}{2} \frac{2}{q + \sqrt{q}} = 1.
    \end{align*}
    Since there are $q$ rows in the matrix $X^*$, the sum of all elements in $X^*$ is $q$. Hence, the objective function value of~\eqref{theta_2} of the solution $X^\prime$ is 
	\begin{align*}
		\langle \mathds{1}_{q \times q}, X^\prime \rangle = \sum_{i=1}^{q} \sum_{j = 1}^q X^\prime_{ij} = \frac{1}{\sqrt{q}}\sum_{i = 1}^q \sum_{j = 1}^q X^*_{ij} =  \frac{1}{\sqrt{q}} q = \sqrt{q}.
	\end{align*}

	From Observation~\ref{theta_paley_graphs} we know that $\vartheta(P_q) = \sqrt{q}$. Therefore, $X^\prime$ is an optimal solution of~\eqref{theta_2} and since $X^*~=~\langle \mathds{1}_{q \times q}, X^\prime \rangle X^\prime$ we know from Lemma~\ref{opt_2} that $(x^*, X^*)$ is an optimal solution of~\eqref{theta_1}.
\end{proof}

\subsection{No improvement up to a certain level}\label{subsection_ESH_Paley}

We now investigate bounds for the stability numbers of Paley graphs that can be obtained by using the ESH approach as described in Section~\ref{method_ESH}. 

For the start, we consider the second and the third level of the ESH, so  $z_k(P_q)$ for $k = 2$ and $k = 3$. For an arbitrary graph $G$, the ESC for a subgraph $G_I$ with~$I = \{i, j\}$ (and thus of order $k = 2$), is equivalent to
\begin{subequations}\label{ESC_2}
\begin{align}
	X_{ij} &\geq 0 \label{2_1} \\
	X_{ij} &\leq X_{ii} \label{2_2} \\
	X_{ij} &\leq X_{jj} \label{2_3} \\
	X_{ii} + X_{jj} &\leq 1 + X_{ij} \label{2_4},
\end{align}
\end{subequations}
while the ESC for a subgraph $G_I$ with $I = \{i, j, \ell \}$ (and thus of order $k = 3$) is equivalent to
\begin{subequations}
	\begin{align}
		X_{ij} + X_{i\ell}&\leq X_{ii} + X_{j\ell} \label{3_1} \\
		X_{ij} + X_{j\ell}&\leq X_{jj} + X_{i\ell} \label{3_2} \\
		X_{i\ell} + X_{j\ell}&\leq X_{\ell \ell} + X_{ij} \label{3_3} \\
		X_{ii} + X_{jj} + X_{\ell\ell} &\leq 1 + X_{ij} + X_{i\ell} + X_{j\ell} \label{3_4}, 
	\end{align}
\end{subequations}
as shown in~\cite{GaarVersionsESH}. 

The relationship between $\vartheta(P_q)$ and $z_k(P_q)$ for $k = 2$ and $k = 3$ can be derived directly from the optimal solution of the SDP for computing the Lovász theta function presented in Lemma~\ref{optimal_solution}, as the next two lemmas show.

\begin{lem}\label{level_2}
	Let $P_q$ be a Paley graph. Then $z_2(P_q) = \vartheta(P_q)$.
\end{lem}

\begin{proof}
Let $(x^*, X^*)$ be as in Lemma~\ref{optimal_solution} and
let $I = \{i,j\} \subseteq V_q$.
Then the inequality~\eqref{2_1} is satisfied, as $X^*_{ij} \geq 0$ by the definition of $X^*$. Furthermore, for all $q$ we have that $\frac{2}{q + \sqrt{q}} \leq \frac{1}{\sqrt{q}}$ and $ 0 \leq \frac{1}{\sqrt{q}}$, so the inequalities~\eqref{2_2} and~\eqref{2_3} also hold for~$X^*$. Finally, since $q \geq 5$, it follows that $X^*_{ii} + X^*_{jj} = x^*_i + x^*_j = \frac{2}{\sqrt{q}} \leq 1$, so the inequality~\eqref{2_4} is satisfied. Hence, adding ESCs for subgraphs of order~$2$ will not improve the Lovász theta function. 
\end{proof}

\begin{lem}\label{level_3}
	Let $P_q$ be a Paley graph with $q\geq 9$. Then $z_3(P_q) = \vartheta(P_q)$.
\end{lem}

\begin{proof}
Let $(x^*, X^*)$ be as in Lemma~\ref{optimal_solution} and 
let $I = \{i,j,\ell\} \subseteq V_q$. 
Then, for the values of $x^*$ and $X^*$ given in Lemma~\ref{optimal_solution}, the inequality
\begin{align*}
X^*_{ij} + X^*_{i\ell} \leq \frac{4}{q + \sqrt{q}} \leq \frac{1}{\sqrt{q}} \leq
x^*_i + X^*_{j\ell} = X^*_{ii} + X^*_{j\ell} 
\end{align*}
holds if $q \geq 9$. Therefore,~\eqref{3_1},~\eqref{3_2} and~\eqref{3_3} hold for $P_q$ with $q \geq 9$. 
Furthermore,
\begin{align*}
X^*_{ii} + X^*_{jj} + X^*_{\ell \ell} = x^*_i + x^*_j + x^*_\ell = \frac{3}{\sqrt{q}} \leq 1   
\leq 1 + X^*_{ij} + X^*_{i\ell} + X^*_{j\ell}
\end{align*}
holds if $q \geq 9$, so~\eqref{3_4} holds.
As a result, $z_3(P_q) = \vartheta(P_q)$ holds for $P_q$ with~$q \geq 9$.
\end{proof}

From Lemmas~\ref{level_2} and \ref{level_3}, we can note that the answer to the question whether $z_k(P_q) = \vartheta(P_q)$ arises from the construction of the optimal solution $(x^*, X^*)$ as stated in Lemma~\ref{optimal_solution}, and thus strongly depends on the value of~$q$. We now consider ESH of an arbitrary level~$k$ for $P_q$ with $q \geq 9$. Since the listing of all possible ESCs is too exhaustive, we consider the convex hulls of subgraphs of fixed order~$k$. Such an approach will enable us to determine levels $k$ of ESH for which $z_k(P_q) = \vartheta(P_q)$. We start by giving the following statement. 

\begin{lem}\label{level_k}
	Let $P_q$ be a Paley graph with $q\geq 9$, let $k \in \mathbb{N}$, let $J_k = \{I \subseteq V_q: \vert I \vert = k\}$ and let $(x^*, X^*)$ be as in Lemma~\ref{optimal_solution}. If $k \leq \frac{\sqrt{q} + 3}{2}$, then 
    $X^{*}_I \in \STAB^{2}(P_{q_I})$ for all~$I \in J_k$.
\end{lem}	

\begin{proof}
    Let $k \leq \frac{\sqrt{q} + 3}{2}$ and let $I =\{v_1, \ldots, v_k\} \in J_k$, i.e., $I \subseteq V_q$ with $\vert I \vert = k$. 
    We have to show that $X^*_I \in \STAB^2(P_{q_I})$ holds. 

    To do so, 
    it is enough to show that 
 there exist non-negative  $\lambda$, $\mu_i$ for all~$1 \leq i \leq k$ 
 and $\nu_{ij}$ for all~ $1\leq i < j \leq k$ 
 such that
	\begin{align}
		\lambda + \sum_{i = 1}^{k} \mu_i + \sum_{i = 1}^{k - 1} \sum_{j = i + 1}^{k} \nu_{ij} = 1 \label{sum_is_one}
	\end{align}
	and 
	\begin{align}
		X^*_I = \lambda 0_k + \sum_{i = 1}^{k} \mu_i E_i + \sum_{i = 1}^{k - 1} \sum_{j = i + 1}^{k} \nu_{ij} E_{ij} \label{matrix_X}
	\end{align}	
    hold, and where $\nu_{ij} = 0$ for all $1\leq i < j \leq k$ with $\{v_i,v_j\} \in E_q$.
 We set
	\begin{alignat}{3}
		\nu_{ij} &= X^*_{v_i,v_j} &&\quad 1\leq i < j \leq k, \nonumber \\
		\mu_i &= x^*_{v_i} - \sum_{\substack{j=1 \\ j \neq i}}^{k} \nu_{{ij}} && \quad 1 \leq i \leq k, \label{coefficient_mu}\\
        \lambda &= 1 - \sum_{i = 1}^{k} \mu_i - \sum_{i = 1}^{k - 1} \sum_{j = i + 1}^{k} \nu_{ij}, \nonumber
	\end{alignat}
	so~\eqref{sum_is_one} is clearly satisfied by construction. Furthermore, \eqref{matrix_X} and $\nu_{ij} = 0$ for all~$1\leq i < j \leq k$ with $\{v_i,v_j\} \in E_q$ hold because of~\eqref{sol_2} and~\eqref{sol_3}, respectively.
 
    From~\eqref{solution} we have that $X^*_{v_i,v_j}$ can either be equal to $0$ or equal to $\frac{2}{q + \sqrt{q}}$; therefore, all $\nu_{ij}$ are non-negative. 
    Regarding $\mu_i$ computed with~\eqref{coefficient_mu}, 
    for a fixed $i$ there are at most $k-1$ values $\nu_{ij}$ 
    equal to $\frac{2}{q + \sqrt{q}}$. Thus,
	\begin{align*}
		\mu_i \geq \frac{1}{\sqrt{q}} - (k-1)\frac{2}{q + \sqrt{q}} = \frac{\sqrt{q} + 3 - 2k}{q + \sqrt{q}},
	\end{align*}
	and since $k \leq \frac{\sqrt{q} + 3}{2}$, all $\mu_i$ are non-negative. 
 
 What is left to show is that $\lambda$ is non-negative. The matrix $X^*_I$ is symmetric, so $X^*_{v_i,v_j} = X^*_{v_j,v_i}$ for all $1 \leq i < j \leq k$. Furthermore, we have that $k \leq \frac{\sqrt{q} + 3}{2}$ and~$q\geq 9$. Therefore,
	\begin{align*}
		\sum_{i = 1}^{k} \mu_i + \sum_{i = 1}^{k - 1} \sum_{j = i + 1}^{k} \nu_{ij} &= \sum_{i = 1}^{k} x^*_{v_i} - \sum_{i = 1}^k \sum_{\substack{j=1 \\ j \neq i}}^{k} X^*_{v_i,v_j} + \sum_{i = 1}^{k - 1} \sum_{j = i + 1}^{k} X^*_{v_i,v_j} \\
		&= \sum_{i = 1}^{k} x^*_{v_i} - 2\sum_{i = 1}^{k - 1} \sum_{j = i + 1}^{k} X^*_{v_i,v_j} + \sum_{i = 1}^{k - 1} \sum_{j = i + 1}^{k} X^*_{v_i,v_j} \\
		&= \sum_{i = 1}^{k} x^*_{v_i} - \sum_{i = 1}^{k - 1} \sum_{j = i + 1}^{k} X^*_{v_i,v_j} 
		\leq \sum_{i = 1}^{k} x^*_{v_i} 
        = k\frac{1}{\sqrt{q}} 
        \leq 1.
	\end{align*}
	Hence, $\lambda$ is non-negative, which finishes the proof. 
\end{proof}

Note that in the proof of Lemma~\ref{level_k} we show that $X^*_I$ is in 
$\STAB^{2}(P_{q_I})$ by showing that it is a convex combination of outer products of incidence vectors of stable sets of at most $2$ vertices. We do not utilize the outer product of incidence vectors of stable sets of more vertices, even though they might exist in $P_q$. The following theorem is a direct consequence of Lemmas~\ref{level_2} and~\ref{level_k}.

\begin{thm}\label{thm_level}
	Let $P_q$ be a Paley graph and let $k  \in \mathbb{N}_0$ with $k \leq \frac{\sqrt{q} + 3}{2}$. Then~$z_k(P_q) = \vartheta(P_q)$.
\end{thm}	

\begin{proof}
If $q = 5$, then $\left\lfloor\frac{\sqrt{q}+3}{2}\right\rfloor = 2$ and $z_2(P_q) = \vartheta(P_q)$ due to Lemma~\ref{level_2}. Thus, the statement holds because of \eqref{ineq:relationship_z}. 
If $q \geq 9$, then the statement follows directly from Lemma~\ref{level_k} and \eqref{ineq:relationship_z}.
\end{proof}

For a Paley graph $P_q$, we define 
\begin{align*}
\ell(q) = \left\lfloor\frac{\sqrt{q}+3}{2}\right\rfloor.
\end{align*}
Clearly, it follows from Theorem~\ref{thm_level} that adding ESCs for subgraphs of orders~$\{2, \dots, \ell(q)\}$ does not improve the Lovász theta function as bound on~$\alpha(P_q)$, so the next corollary holds.

\begin{cor}\label{cor_level}
	Let $P_q$ be a Paley graph.
    Then $z_k(P_q) = \vartheta(P_q)$ holds for all~$k  \in \mathbb{N}_0$ with $k \leq \ell(q)$.
\end{cor}	

So far, we do not know anything about adding ESCs for subgraphs of order~$\ell(q) +1$ besides in the case when $q$ is square with Observation~\ref{theta_exact_q_square}. For some Paley graphs, we might have that $z_{\ell(q)}(P_q) = z_{\ell(q)+1}(P_q)$, while for others, we might have that adding ESCs of size $\ell(q) + 1$ yield bounds that are tighter than the Lovász theta function. In the next section, we will perform computational experiments and empirically determine levels for which imposing ESCs improves the Lovász theta function of a Paley graph~$P_q$. Additionally, we will confirm computationally that adding ESCs for subgraphs of order up to $\ell(q)$ does not yield tighter bounds.

\subsection{Computational study}\label{comp_justification}

Next, we want to investigate if there is an improvement of the Lovász theta function as bound on $\alpha(P_q)$ when including ESCs for subgraphs of orders $\ell(q) + 1$ and higher. To do so, we first conduct a computational study 
and compute upper bounds on $z_k(P_q)$ for $k \in \{2, \dots, 10\}$ and $q < 200$.
Note that we do not consider any $q$ that is a square in this computational study because we already know what happens on all levels of the ESH due to  Observation~\ref{theta_exact_q_square}.


In order to compute upper bounds on $z_k(P_q)$ with MATLAB R2022b we use the source code~\href{https://gitlab.com/egaar/tightenthetafunction}{https://gitlab.com/egaar/tightenthetafunction}, which accompanies~\cite{Gaa:18} and lead to the publications~\cite{Gaa:20,GaarSiebenhoferWiegeleStabBaB}.
We call the function \texttt{tightenThetaFunction()} for the \texttt{adjacencyMatrix} of the Paley graph~$P_q$, which iteratively adds 
violated ESC for subgraphs of order~$k$ in each of $\texttt{nrIterations} = 10$ cycles and thus computes an upper bound on $z_k(P_q)$ by computing $z_J(P_q)$ for growing sets $J$ of subgraphs of order~$k$. To find violated ESC to add to $J$, in every cycle the function performs a simulated annealing heuristic for $\texttt{nrExecutionsSimAnn} = 4000$ times, where it tries to minimize the inner product between $X_I$ and several different separation matrices (stored in $\texttt{usedSeparationMatricesPerRun}$) determined with the function \texttt{setupSeparation\-Matrices()}
with all ($\texttt{mySMVersion} = 7$, $\texttt{usedMethodToChangeNrSimAnnPerSepMat} = 1$) possible separation matrices from the function \texttt{setupPossibleSeparationMatrices()}, which includes extreme copositive matrices with entries in $\{-1,0,1\}$, facet inducing matrices of the squared stable set polytope and random matrices, see~\cite{Gaa:18,Gaa:20} for more details.
Then the code includes the $\texttt{maxNr\-NewViolatedSubgraphsPer}$ $\texttt{Iteration} = 1000$ most violated ESC found by adding one inequality constraint inducing a separating hyperplane as described in~\cite{GaarSiebenhoferWiegeleStabBaB} 
(\texttt{usedCutting\-Planes\-PerRun} $= 3$). Each SDP is solved with MOSEK~10.0.~\cite{mosek} ($\texttt{usedSolver}$ $\texttt{PerRun} = 7$). All computations were done on an Intel(R) Core(TM) i7-1260P CPU @ 2.10GHz with 64GB of RAM, operated under Windows 10.  The code and the instances are available as ancillary files from the arXiv page of this paper: \href{https://arxiv.org/abs/arXiv:2412.12958}{arXiv:2412.12958}.

The results, shown in Table~\ref{Table_2}, are organized as follows. First, we give basic information for every considered Paley graph $P_q$: the value of $q$, the stability number $\alpha(P_q)$, the value of the Lovász theta function $\vartheta(P_q)$, and the value of $\ell(q)$. Note that the values of $\alpha(P_q)$ are taken from~\cite{Gvozdenovic2009} for $61 \leq q < 200$, where $q$ is a prime, and from \url{https://aeb.win.tue.nl/graphs/Paley.html} for $q < 61$ and $q = 125$. Furthermore, we report computed upper bounds on $z_k(P_q)$ for $k \in \{2, \dots, 10\}$. The respective cell for $z_k(P_q)$ is shaded in gray if $k \leq \ell(q)$. If for some $P_q$ there exists a level~$k > \ell(q)$ such that the upper bound on $z_k(P_q)$ is equal to $\vartheta(P_q)$, we highlight by boldface the value in the respective cell. 
Finally, based on the computed data, we determine the empirical maximum level~$\tilde{\ell}(q)$ for which adding the ESCs does not improve the Lovász theta function as bound on the stability number, i.e., the upper bound on $z_{\tilde{\ell}(q)}$ is equal to $\vartheta(P_q)$ and the upper bound on $z_{\tilde{\ell}(q)+1}$ is less than $\vartheta(P_q)$, so on level~$\tilde{\ell}(q)+1$ there is the first improvement over the Lovász theta function.

Note that as we only compute upper bounds on~$z_k(P_q)$ by computing $z_J(P_q)$ for some set~$J$, it is possible that for some values of $q$ the values of the upper bounds on $z_k(P_q)$ increase for higher values of $k$, even though the actual values of $z_k(P_q)$ decrease. 
Whenever we know that the computed upper bound on $z_k(P_q)$ is larger than $z_k(P_q)$ (because for some $k'<k$ the computed upper bound on $z_{k'}(P_q)$ is lower than the computed upper bound on $z_{k}(P_q)$), we highlight the corresponding upper bound on $z_k(P_q)$ in italic in Table~\ref{Table_2}. This non-monotonic behavior of the upper bounds on $z_k(P_q)$ could appear because for the values of $k$ up to $6$ a lot of investigation has been done in cleverly choosing separation matrices for finding violated ESC and thus the computed upper bounds on $z_k(P_q)$ will be more likely close to the actual values of $z_k(P_q)$; see~\cite{Gaa:18} for more details.
Due to the fact that we computed only upper bounds on $z_k(P_q)$ for determining the empirical level~$\tilde{\ell}(q)$, it is possible that actually there is already improvement of the bound over the Lovász theta function on level~$\tilde{\ell}(q)$ or lower.


\begin{table}[hbt!]
	\caption{Computed upper bounds on $z_k(P_q)$ 
 for Paley graphs $P_q$ 
 }
	\centering
    \begin{adjustbox}{max width=\textwidth}
	\begin{tabular}{ r r r | r r r r r r r r r r r}
		$q$ & $\alpha(P_q)$ & $\vartheta(P_q)$ & $\ell(q)$ & $\tilde{\ell}(q)$ & $z_2(P_q)$ & $z_3(P_q)$ & $z_4(P_q)$ & $z_5(P_q)$ & $z_6(P_q)$ & $z_7(P_q)$ &$z_8(P_q)$ & $z_9(P_q)$ & $z_{10}(P_q)$ \\
		\hline
        5 & 2 & 2.2361 & 2 & 2 & \cellcolor{lightgray} 2.2361 & 2.0000 & 2.0000 & 2.0000 & - & - & - & - & - \\
        13 & 3 & 3.6056 & 3 & 3 & \cellcolor{lightgray} 3.6056 & \cellcolor{lightgray} 3.6056 & 3.0000 & 3.0000 & 3.0000 & 3.0000 & 3.0000 & 3.0000 & 3.0000 \\
        17 & 3 & 4.1231 & 3 & 3 & \cellcolor{lightgray} 4.1231 & \cellcolor{lightgray} 4.1231 & 3.6651 & 3.6646 & 3.6411 & 3.6404 & 3.2846 & 3.1512 & 3.0674 \\
        29 & 4 & 5.3852 & 4 & 4 & \cellcolor{lightgray} 5.3852 & \cellcolor{lightgray} 5.3852 & \cellcolor{lightgray} 5.3852 &  4.6187 &  4.6122 &  4.4966 &  4.4966 &  4.4963 &  4.4958 \\
        37 & 4 & 6.0828 & 4 & 4 & \cellcolor{lightgray} 6.0828 & \cellcolor{lightgray} 6.0828  & \cellcolor{lightgray} 6.0828 &  5.6048 &  \textit{5.6369} &  5.4941 &  5.4935 &  5.1477 &  \textit{5.1896} \\
        41 & 5 & 6.4031 & 4 & 4 & \cellcolor{lightgray} 6.4031 & \cellcolor{lightgray} 6.4031  & \cellcolor{lightgray} 6.4031 &  6.0702 &  \textit{6.1107} &  5.9933 &  5.9925 &  5.5562 &  \textit{5.5703} \\
        53 & 5 & 7.2801 & 5 & 5 & \cellcolor{lightgray} 7.2801 & \cellcolor{lightgray} 7.2801 & \cellcolor{lightgray} 7.2801 & \cellcolor{lightgray} 7.2801 &  6.9560 &  6.1956 &  6.1915 &  \textit{6.1945} &  \textit{6.1942} \\
		61 & 5 & 7.8102 & 5 & 5 & \cellcolor{lightgray} 7.8102 & \cellcolor{lightgray} 7.8102 &  \cellcolor{lightgray} 7.8102 & \cellcolor{lightgray} 7.8102 &  7.6664 &  7.0000 &  6.9963  &  \textit{6.9988}  &  \textit{6.9978}  \\
		73 & 5 & 8.5440 & 5 & 5 & \cellcolor{lightgray} 8.5440 & \cellcolor{lightgray} 8.5440 & \cellcolor{lightgray} 8.5440 & \cellcolor{lightgray} 8.5440 &  8.5297  &  8.1970  &  8.1910  &  \textit{8.1949} &  8.1386  \\
		89 & 5 & 9.4340 & 6 & 8 & \cellcolor{lightgray} 9.4340 & \cellcolor{lightgray} 9.4340 & \cellcolor{lightgray} 9.4340 & \cellcolor{lightgray} 9.4340 & \cellcolor{lightgray} 9.4340 & \textbf{9.4340} & \textbf{9.4340} &  9.2172 &  \textit{9.2651}  \\
		97 & 6 & 9.8489 & 6 & 6 & \cellcolor{lightgray} 9.8489 & \cellcolor{lightgray} 9.8489 & \cellcolor{lightgray} 9.8489 & \cellcolor{lightgray} 9.8489 & \cellcolor{lightgray} 9.8489 &  9.3042 &  9.0777 &  \textit{9.1612} &  \textit{9.3810}\\
		101 & 5 & 10.0499 & 6 & 9 & \cellcolor{lightgray} 10.0499 & \cellcolor{lightgray} 10.0499 &  \cellcolor{lightgray} 10.0499 &  \cellcolor{lightgray} 10.0499 &  \cellcolor{lightgray} 10.0499 & \textbf{10.0499} & \textbf{10.0499} & \textbf{10.0499} &  10.0415 \\
		109 & 6 & 10.4403 & 6 & 6 & \cellcolor{lightgray} 10.4403 & \cellcolor{lightgray} 10.4403 & \cellcolor{lightgray} 10.4403 & \cellcolor{lightgray} 10.4403 & \cellcolor{lightgray} 10.4403  &  10.1091  &  10.0652  &  \textit{10.1459} &  \textit{10.3095} \\
		113 & 7 & 10.6301 & 6 & 6 & \cellcolor{lightgray} 10.6301 & \cellcolor{lightgray} 10.6301 & \cellcolor{lightgray} 10.6301 & \cellcolor{lightgray} 10.6301 & \cellcolor{lightgray} 10.6301 &  10.3927 &  10.3732 &  \textit{10.4730} &  \textit{10.5898} \\
        125 & 7 & 11.1803 & 7 & 7 & \cellcolor{lightgray} 11.1803 & \cellcolor{lightgray} 11.1803 & \cellcolor{lightgray} 11.1803 & \cellcolor{lightgray} 11.1803 & \cellcolor{lightgray} 11.1803 & \cellcolor{lightgray} 11.1803 & 10.6415 & \textit{10.7220} & \textit{10.9749} \\
        137 & 7 & 11.7047 & 7 & 7 & \cellcolor{lightgray} 11.7047 & \cellcolor{lightgray} 11.7047 & \cellcolor{lightgray} 11.7047 & \cellcolor{lightgray} 11.7047 & \cellcolor{lightgray} 11.7047 & \cellcolor{lightgray} 11.7047 &  11.2181 &  \textit{11.3131} &  \textit{11.5605} \\ 
        149 & 7 & 12.2066 & 7 & 7 & \cellcolor{lightgray} 12.2066 & \cellcolor{lightgray} 12.2066 & \cellcolor{lightgray} 12.2066 & \cellcolor{lightgray} 12.2066 & \cellcolor{lightgray} 12.2066 & \cellcolor{lightgray} 12.2066 &  11.9771 &  \textit{12.1427} & \textit{12.1907} \\
        157 & 7 & 12.5300 & 7 & 7 & \cellcolor{lightgray} 12.5300 & \cellcolor{lightgray} 12.5300 & \cellcolor{lightgray} 12.5300 & \cellcolor{lightgray} 12.5300 & \cellcolor{lightgray} 12.5300 & \cellcolor{lightgray} 12.5300 &  12.4315 &  \textit{12.5108} &  \textit{12.5259} \\
        173 & 8 & 13.1529 & 8 & 8 & \cellcolor{lightgray} 13.1529 & \cellcolor{lightgray} 13.1529 & \cellcolor{lightgray} 13.1529 & \cellcolor{lightgray} 13.1529 & \cellcolor{lightgray} 13.1529 & \cellcolor{lightgray} 13.1529 & \cellcolor{lightgray} 13.1529 &  13.0004 &  \textit{13.1152} \\
        181 & 7 & 13.4536 & 8 & $\geq 10$ & \cellcolor{lightgray} 13.4536 & \cellcolor{lightgray} 13.4536 & \cellcolor{lightgray} 13.4536 & \cellcolor{lightgray} 13.4536 & \cellcolor{lightgray} 13.4536 & \cellcolor{lightgray} 13.4536 & \cellcolor{lightgray} 13.4536  & \textbf{13.4536}  & \textbf{13.4536} \\
        193 & 7 & 13.8924 & 8 & $\geq 10$ & \cellcolor{lightgray} 13.8924 & \cellcolor{lightgray} 13.8924 & \cellcolor{lightgray} 13.8924 & \cellcolor{lightgray} 13.8924 & \cellcolor{lightgray} 13.8924 & \cellcolor{lightgray} 13.8924 & \cellcolor{lightgray} 13.8924 & \textbf{13.8924} & \textbf{13.8924} \\
        197 & 8 & 14.0357 & 8 & 8 & \cellcolor{lightgray} 14.0357 & \cellcolor{lightgray} 14.0357 & \cellcolor{lightgray} 14.0357 & \cellcolor{lightgray} 14.0357 & \cellcolor{lightgray} 14.0357 & \cellcolor{lightgray} 14.0357 & \cellcolor{lightgray} 14.0357 &  13.9961 & \textit{14.0289} \\
		\hline
	\end{tabular}
    \end{adjustbox}
	\label{Table_2}
\end{table}

From the results in Table~\ref{Table_2}, we first observe that  including ESCs on subgraphs of orders $\{2, \ldots, \ell(q)\}$ (in the gray cells) does not improve the Lovász theta function for all $22$ considered $P_q$, which aligns with Corollary~\ref{cor_level}. So our computational results fit the theoretical results obtained so far, and in particular, $\tilde{\ell}(q) \geq \ell(q)$ holds for all $P_q$.
However, for some values of $q$, we see that $\tilde{\ell}(q) > \ell(q)$ holds, i.e., there is still no improvement of the Lovász theta function also on levels of the exact subgraph hierarchy that are higher than $\ell(q)$. We will continue to investigate these cases in Section~\ref{sec_relationship_alpha}.

Generally, the considered upper bounds on the levels of the ESH yield good bounds on the stability numbers for small Paley graphs. For $P_q$ with $q \in \{5, 13\}$, the bounds are best possible and coincide with the stability number, while for $q \in \{17, 29, 41\}$ rounding of the best bounds yields the stability numbers. For $P_q$ with $q\in\{37, 53, 61,125, 149, 197\}$, the best computed bounds are better than the Lovász theta function by one integer. 

For the remaining larger Paley graphs, better integer bounds were not obtained.
We will investigate in Section~\ref{section_ESH_local} how even better bounds on the stability number of Paley graphs can be obtained.

\subsection{No improvement on higher levels}
\label{sec_relationship_alpha}

So far, we know that for a Paley graph $P_q$, there is no improvement of the Lovász theta function as bound on the stability number up to level~$\ell(q)$ of the ESH. It is the aim of this section to investigate what happens on higher levels of the ESH.

To do so, we first go back to the computational results.
For most of the considered Paley graphs $P_q$, the values of $\ell(q)$ and the computational empirical level up until there is no improvement $\tilde\ell(q)$ given in Table~\ref{Table_2} coincide, which suggests that for these graphs indeed $\ell(q)$ is the largest order of subgraphs, for which adding the ESCs does not improve the Lovász theta function as bound on the stability number. 
However, for $q \in \{89,101,181,193\}$ the values of $\tilde\ell(q)$ are larger then $\ell(q)$. 

Interestingly, the stability number $\alpha(P_q)$ appears to influence the relationship between $\ell(q)$ and $\tilde\ell(q)$. For Paley graphs where $\alpha(P_q) \geq \ell(q)$, we note that $\ell(q) = \tilde{\ell}(q)$ for all considered values of~$q$, thus adding ESCs for subgraphs of order~$\ell(q) + 1$ led to better bounds.
In contrast to that, for the four considered Paley graphs with $\alpha(P_q) < \ell(q)$, i.e., $q \in \{89,101,181,193\}$, this enhancement is not achieved, leading to $\ell(q) < \tilde{\ell}(q)$.
We visualize the values of $\alpha(P_q)$, $\ell(q)$ and $\tilde{\ell}(q)$ (where we plot the value 10 as $\tilde{\ell}(q)$ for $q \in \{181, 193\}$, even though this is actually just a lower bound on the unknown value of $\tilde{\ell}(q)$) in Figure~\ref{figure_1}.

\begin{figure}[hbt!]
	\centering 
 \caption{The values of $\alpha(P_q)$, $\ell(q)$, and $\tilde{\ell}(q)$ for Paley graphs $P_q$ 
 } 
\begin{tikzpicture}[scale=1.0]
			\begin{axis}[
				xlabel={$q$},
				xmin=0, xmax=200,
				ymin=0, ymax=11,
				xtick={25, 50, 75, 100, 125, 150, 175, 200},
				ytick={0,1, 2, 3, 4, 5, 6, 7, 8, 9, 10, 11},
				legend pos=north west,
				ymajorgrids=true,
				grid style=dashed,
				]			
				\addplot[
				color=blue,
				mark=square,
				]
				coordinates {
					(5,2) (13, 3) (17,3) (29, 4) (37, 4) (41, 5) (53, 5) (61, 5) (73, 5) (89, 5) (97, 6) (101, 5) (109, 6) (113, 7) (125, 7) (137,7) (149, 7) (157, 7) (173, 8) (181, 7) (193, 7) (197, 8)
				};
            \addlegendentry{$\alpha(P_q)$}
            \addplot[
				color=red,
				mark=triangle*,
				]
				coordinates {
					(5, 2) (13, 3)  (17,3) (29, 4) (37, 4) (41, 4) (53, 5) (61, 5) (73, 5) (89, 6) (97, 6) (101, 6) (109, 6) (113, 6) (125, 7) (137,7) (149, 7) (157, 7) (173, 8) (181, 8) (193, 8) (197, 8)
				};
            \addlegendentry{$\ell(q)$}
            \addplot[
				color=teal,
				mark=asterisk,
				]
				coordinates {
					(5, 2) (13, 3)  (17,3) (29, 4) (37, 4) (41, 4) (53, 5) (61, 5) (73, 5) (89, 8) (97, 6) (101, 9) (109, 6) (113, 6) (125, 7) (137,7) (149, 7) (157, 7) (173, 8) (181, 10) (193, 10) (197, 8)
				};
            \addlegendentry{$\tilde{\ell}(q)$}
			\end{axis}
		\end{tikzpicture}
	\label{figure_1}
\end{figure}
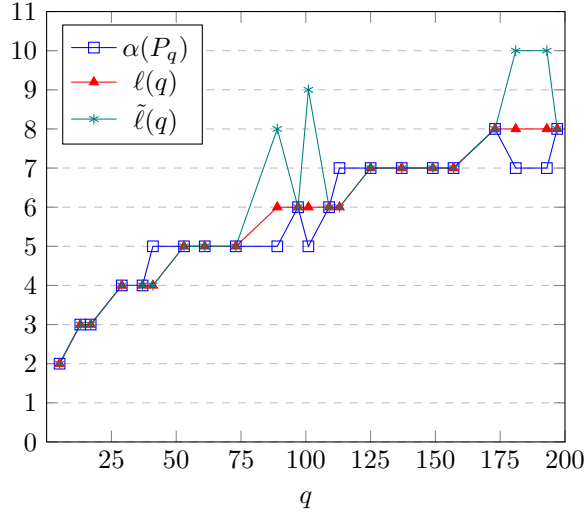

Indeed, it turns out that the fact that $\tilde{\ell}(q) > \ell(q)$ holds whenever $\alpha(P_q) < \ell(q)$ is satisfied is not a coincidence, as the following result shows.

\begin{thm}\label{thm_alpha_levelPlus1}
	Let $P_q$ be a Paley graph with $q\geq 25$.
    If $\alpha(P_q) < \ell(q)$, then~$z_{\ell(q)+1}(P_q) = z_{\ell(q)}(P_q) = \vartheta(P_q)$.
\end{thm}	
\begin{proof}
    Let $P_q$ be a Paley graph such that $\alpha(P_q) < \ell(q)$ holds and
    let $(x^*, X^*)$ be as in Lemma~\ref{optimal_solution}.
    Let $H$ be a subgraph of $P_q$ on $k=\ell(q)+1$ vertices~$\{v_1, \dots, v_k\}$. Due to Corollary~\ref{cor_level} it is enough to show that the ESC for $H$ is fulfilled. Towards this end we consider two cases.

    Case 1: 
    If all vertices of $H$ have at least one neighbor in $H$, then in each row of~$X^*_H$ there is at least one off-diagonal zero entry according to~\eqref{solution} in Lemma~\ref{optimal_solution}. This implies that in each row of $X^*_H$ there are at most $k-2 = \ell(q)-1$ off-diagonal entries equal to $\frac{2}{q+\sqrt{q}}$.
    With this property in mind, the steps of the proof of  Lemma~\ref{level_k} can be followed in order to finish the proof. In particular, the values of $\lambda$, $\mu_i$, and $\nu_{ij}$ are set in the same fashion, i.e., as they are set in \eqref{coefficient_mu}. The above property implies that all $\mu_i$ are non-negative, and in order to guarantee $\lambda \geq 0$ one has to have $k\frac{1}{\sqrt{q}} \leq 1$, which is fulfilled for all $q \geq 25$.

    Case 2:
    If there is at least one vertex of $H$ that does not have a neighbor in~$H$, then let without loss of generality $v_1, \dots, v_p$ 
    be the vertices of $H$ that do not have any neighbor in~$H$. 
    If $p \geq 3$, then we define $S = \{v_1, \dots, v_p\}$. 
    If $p \leq 2$, 
    then due to Observation~\ref{alpha_equals_clique} we know that $\omega(P_q) = \alpha(P_q) < \ell(q)$, so there are two distinct vertices~$v^*$ and $v^{**}$ among the $\ell(q)$ vertices $\{v_2, \dots, v_k\}$ that are not adjacent and we define $S = \{v_1, \dots, v_p\} \cup \{v^*, v^{**}\}$.
    So independent from the value of $p$, the set~$S$ is a stable set in $H$ on at least three vertices, and all vertices in $H$ that have no neighbor in $H$ are in $S$.
    Let $s \in \{0,1\}^k$ be the incidence vector of the stable set $S$ in $H$. 
    
    In order to show that the ESC for $H$ is satisfied, 
    we set 
	\begin{alignat}{3}
        \sigma &= \frac{2}{q+\sqrt{q}}, \nonumber \\
		\nu_{ij} &= \left\{
\begin{array}{ll}
0 & \text{ if } v_i \in S \text{ and } v_j \in S\\
X^*_{v_i,v_j} & \text{ otherwise}
\end{array}
\right.
 &&\quad 1\leq i < j \leq k, \nonumber \\
		\mu_i &= 
\left\{
\begin{array}{ll}
x^*_{v_i} - \sigma - \sum_{\substack{j=1 \\ j \neq i}}^{k} \nu_{{ij}} & \text{ if } v_i \in S \\
x^*_{v_i} - \sum_{\substack{j=1 \\ j \neq i}}^{k} \nu_{{ij}} & \text{ otherwise}
\end{array}
\right.  
   && \quad 1 \leq i \leq k, \nonumber\\
        \lambda &= 1 - \sum_{i = 1}^{k} \mu_i - \sum_{i = 1}^{k - 1} \sum_{j = i + 1}^{k} \nu_{ij} - \sigma , \nonumber
	\end{alignat}
 and see that
	\begin{align*}
		&\lambda + \sum_{i = 1}^{k} \mu_i + \sum_{i = 1}^{k - 1} \sum_{j = i + 1}^{k} \nu_{ij} + \sigma = 1, \\
		&X^*_H = \lambda 0_k + \sum_{i = 1}^{k} \mu_i E_i + \sum_{i = 1}^{k - 1} \sum_{j = i + 1}^{k} \nu_{ij} E_{ij} + \sigma ss^T, \
	\end{align*}	
    and $\nu_{ij} = 0$ for all $1\leq i < j \leq k$ with $\{v_i,v_j\} \in E_q$ hold.
    Furthermore, clearly $\sigma$ and all $\nu_{ij}$ are non-negative by construction.

    Additionally,  
    \begin{align*}
        \mu_i \geq \frac{1}{\sqrt{q}} - (k-2)\frac{2}{q + \sqrt{q}} 
    \end{align*}
    holds for all values of $1 \leq i \leq k$. 
    In particular, 
    if $v_i \in S$, then by construction ($S$ hast at least three elements) at least two values of $\nu_{ij}$ for $1 \leq j \leq k$, $j \neq i$ are zero, and $\sigma$ has the same value as the remaining values of $\nu_{ij}$. 
    If $v_i \notin S$, then  $v_i$ has at least one neighbor in $H$ and hence at least one of $\nu_{ij}$ for $1 \leq j \leq k$, $j \neq i$ is zero.
    As a consequence, all values of $\mu_i$ are non-negative since  $k = \ell(q) + 1 = \lfloor \frac{\sqrt{q} + 5}{2} \rfloor$.

    Finally, $\lambda \geq 0$ holds by analogous arguments as in the proof of Lemma~\ref{level_k} as soon as $k\frac{1}{\sqrt{q}} \leq 1$ holds, which is the case for all $q \geq 25$. As a consequence, the ESC for $H$ is fulfilled for $p \geq 3$.
\end{proof}

Note that in the proof of Theorem~\ref{thm_alpha_levelPlus1} we show that all ESCs for subgraphs of the order~$\ell(q) + 1$ are satisfied by considering outer products of incidence vectors of stable sets, where all but one stable set have at most $2$ vertices and one stable set has at least $3$ vertices. Thus, it is not surprising that we are able to prove a stronger result than Lemma~\ref{level_k}, where in the proof we only use outer products of incidence vectors of stable sets of at most $2$ vertices. 

As a result of Theorem~\ref{thm_alpha_levelPlus1}, if $\alpha(P_q) < \ell(q)$ holds, then there is not only no improvement of the ESH up until level~$\ell(q)$, but also on the next level~$\ell(q)+1$. 

When we again turn to the computational results in Table~\ref{Table_2}, then for the four 
instances with $\alpha(P_q) < \ell(q)$, i.e., $q \in \{89,101,181,193\}$, also including ESCs for subgraphs of order~$\ell(q) + 2$ fails to yield any improvement, suggesting that the minimum level of the ESH where an improvement may be achieved is $\ell(q) + 3$ for these values of $q$. 
It remains an open question whether indeed, if $\alpha(P_q) < \ell(q)$ holds, there is no improvement up to level $\ell(q) + 2$ for all values of $q$, and it is still open to prove this for $q \in \{89,101,181,193\}$. Most likely, such proofs need to consider outer products of incidence vectors of stable sets on more vertices or more of these outer products. 

\section{The vertex-transitive exact subgraph hierarchy}\label{section_ESH_local}

In this section, we introduce a new hierarchy for the stable set problem for vertex-transitive graphs, which builds upon the ESH and which we will call the vertex-transitive ESH. 

\subsection{Definition}\label{intro_LESH}

Let $G = (V,E)$ with $\vert V \vert = n$ and $V = \{0, \dots, n-1\}$ be a vertex-transitive graph. Since $G$ is vertex-transitive, we can, without loss of generality, choose the vertex $0$ to be in a maximum stable set of $G$. Therefore, any other vertex adjacent to the vertex $0$ cannot be contained in the maximum stable set. Consequently, all other vertices contained in the maximum stable set are elements of $V^L = V \setminus (\{0\} \cup N(0))$. We call the induced subgraph of $G$ induced by the vertices $V^L$ the local graph of $G$ and denote it by $G^L$. 
This terminology is inspired by the nomenclature used by Magsino, Mixon, and Parshall~\cite{Magsino} for Paley graphs, which we will discuss in more detail in Section~\ref{LESH_Paley} that deals with Paley graphs.
From the given argumentation, we note the following.

\begin{obs}\label{stability_local_graph}
Let $G = (V,E)$ with  $\vert V \vert = n$ and $V = \{0, \dots, n-1\}$ be a vertex-transitive graph, and let $G^L$ be the local graph of $G$. Then,
\begin{align*}
    \alpha(G) = 1 + \alpha(G^L).
\end{align*}
\end{obs}

Hence, calculating the stability number of a vertex-transitive graph $G$ reduces to computing the stability number of $G^L$, and we obtain a maximum stable set of $G$ by taking a maximum stable set in $G^L$ and adding the vertex~$0$.

Since every vertex-transitive graph is $r$-regular for some $r \in \mathbb{N}_0$, an immediate advantage of Observation~\ref{stability_local_graph} is that the order of the graph that is considered reduces from $n$ for $G$ to $n - r - 1$ for $G^L$. However, the computation of the stability number of $G^L$, and thus $G$, remains NP-hard. 

Therefore, we introduce the vertex-transitive ESH to explore upper bounds on $\alpha(G)$ using the following principle. Following the theory of the ESH from Section~\ref{method_ESH}, we know that when the constraint $X \in \STAB^{2}(G^L)$ is added to the SDP $\eqref{theta_1}$ for $G^L$ to compute the Lovász theta function $\vartheta(G^L)$, the optimal objective function value of the resulting SDP is $\alpha(G^L)$. Due to Observation~\ref{stability_local_graph}, this allows us to deduce the stability number of $G$. Nevertheless, the obtained SDP is still too large to solve, so we include the constraint $X \in \STAB^{2}(G^L)$ only partially for smaller subgraphs. In particular, for~$k \in \{0, \ldots, n - r - 1\}$ we define the $k$-th level of the vertex-transitive ESH (VTESH) as
\begin{align}\label{k_th_level_LESH}
	z^\prime_k(G) &= 1 + z_k(G^L).
\end{align}
Furthermore, we define $z^\prime_k(G) = z^\prime_{n-r-1}(G)$ for $k > n - r - 1$, analogously to the definition of the ESH.

From Observation~\ref{stability_local_graph} we obtain the relationship 
\begin{align}
\label{ineq:relationship_zPrime}
\alpha(G) = 1 + \alpha(G^L) = z^\prime_{n - r - 1}(G) \leq \cdots \leq z^\prime_2(G) \leq z^\prime_1(G) = z^\prime_0(G) = 1 + \vartheta(G^L),
\end{align}
so the higher the considered level~$k$ of the VTESH of a vertex-transitive graph is, 
the better the bounds on the stability number are, just as it is the case for the ESH.

\subsection{Relationship between hierarchies}\label{relationship_ESH_LESH}

As the next step, we investigate the relationship between the bounds on the $k-$th level of the ESH and the VTESH for vertex-transitive graphs. We begin by establishing the following statement.

\begin{lem}\label{ESH_subgraphs}
Let $G = (V, E)$ be a graph on $n$ vertices, let $I\subseteq V$, $I \neq \emptyset$, and let $(x, X)$ be a feasible solution for~\eqref{theta_1}. Let $i \in I$ such that 
\begin{enumerate}[(a)]
\item\label{first} $X_{ij} = 0$ for all $j \in I$, or
\item\label{second} $X_{ii} = 1$ and $X_{ij} = X_{jj}$ for all $j \in I\setminus \{i\}$
\end{enumerate}
holds. Then
\begin{align*}
X_I \in \STAB^2(G_I) \quad \Leftrightarrow \quad X_{I \setminus \{i\}} \in \STAB^2(G_{I \setminus \{i\}}).
\end{align*}
\end{lem}

\begin{proof}
We proceed by proving both directions of the equivalence separately.

"\(\Rightarrow\)": Suppose \( X_I \in \mathrm{STAB}^2(G_I) \). Gaar~\cite[Observation 2.1.5]{Gaa:18} showed that if an ESC is fulfilled for a subgraph $H$, it is also satisfied for each subgraph of $H$. Thus, $X_{I \setminus \{i\}} \in \STAB^2(G_{I \setminus \{i\}})$ holds.

"\(\Leftarrow\)": Suppose $X_{I \setminus \{i\}} \in \STAB^2(G_{I \setminus \{i\}})$. Without loss of generality, we assume that $i$ is the last entry of $I$.
Let $k=|I|$ and
let $\mathcal{S}(G_{I \setminus \{i\}})$ be the set of all incidence vectors of stable sets in $G_{I \setminus \{i\}}$. Then, there exist some $t \geq 0$, some $s_\ell \in \mathcal{S}(G_{I \setminus \{i\}})$ and some $\lambda_{\ell} > 0$ for all $1 \leq \ell \leq t$, such that 
\begin{align}\label{sum_of_coef}
	\sum_{\ell = 1}^{t} \lambda_{\ell} = 1
\end{align}
and such that the matrix $X_{I \setminus \{i\}}$ can be expressed as 
\begin{align}\label{matrix_X_I_without_i}
	X_{I \setminus \{i\}} = \sum_{\ell = 1}^{t} \lambda_\ell s_{\ell} s^T_{\ell}.
\end{align}
We now show that $X_{I} \in \mathrm{STAB}^2(G_{I})$ by considering two cases, depending on whether~\eqref{first} or~\eqref{second} holds for $i$.

Case 1: Assume~\eqref{first} holds, so $X_{i j} = 0$ for all $j \in I$. In this case, every incidence vector of a stable set $s_{\ell} \in \mathcal{S}(G_{I \setminus \{i\}})$ can be extended to an incidence vector of a stable set $\tilde{s}_{\ell} \in \mathcal{S}(G_{I})$ by setting $\tilde{s}_{\ell} = \begin{pmatrix} s_{\ell} & 0 \end{pmatrix}^T$. 

We set $\tilde\lambda_{\ell} = \lambda_{\ell}$ for all $1 \leq \ell \leq t$. Then, $\tilde\lambda_{\ell} \geq 0$ for all $1 \leq \ell \leq t$, and 
$\sum_{\ell = 1}^{t} \lambda^\prime_{\ell} = 1$
due to~\eqref{sum_of_coef}.
Furthermore, given that because of~\eqref{first}
\begin{align*}
X_{I} = \begin{pmatrix} X_{I \setminus \{i\}} & 0_{k-1} \\ 0^T_{k-1} & 0 \end{pmatrix}
\end{align*}
and 
\begin{align*}
    \tilde{s}_\ell \tilde{s}^T_\ell = \begin{pmatrix} s_\ell s^T_\ell & 0_{k-1} \\ 0^T_{k-1} & 0 \end{pmatrix},
\end{align*}
we obtain that we can write $X_I$ as
\begin{align*}
	X_{I} = \sum_{\ell = 1}^{t} \tilde{\lambda}_{\ell} \tilde{s}_{\ell} \tilde{s}^T_{\ell}. 
\end{align*}
Hence, $X_I \in \STAB^2(G_I)$ holds.

Case 2: Assume~\eqref{second} holds, so $X_{i i} = 1$ and $X_{i j} = X_{jj}$ for all $j \in I \setminus \{i\}$. 

First, we argue that $s^\prime_{\ell} = \begin{pmatrix} s_{\ell} & 1 \end{pmatrix}^T$ is an incidence vector of a stable set in $G_I$ for all $1 \leq \ell \leq t$.
Indeed, assume that this is not the case. Then, there exists some~$\ell^\prime$ such that $s^\prime_{\ell^\prime} = \begin{pmatrix} s_{\ell^\prime} & 1 \end{pmatrix}^T$ is not an incidence vector of a stable set in $G_I$. Hence, there exists some $j^\prime \in I \setminus \{i\}$ such that~$[s_{\ell^\prime}]_{j^\prime} = 1$ and $\{i, j^\prime\} \in E$. Now, since~$\{i, j^\prime\} \in E$, we have that $X_{i{j^\prime}} = 0$, and therefore~$X_{j^\prime j^\prime} = 0$ because of~\eqref{second}. Due to the representation~\eqref{matrix_X_I_without_i} and because $\lambda_\ell > 0$ and $s_{\ell} \in \{0,1\}^{|I|-1}$ for all $1 \leq \ell \leq t$, this implies $\lambda_{\ell}[s_{\ell}]_{j^\prime} = 0$ for all $1 \leq \ell \leq t$, and hence $[s_{\ell^\prime}]_{j^\prime} = 0$, a contradiction. Therefore, $s^\prime_{\ell}$ is indeed an incidence vector of a stable set in $G_I$ for all $1 \leq \ell \leq t$.

Now, we set $\lambda^\prime_{\ell} = \lambda_{\ell}$ for all $1 \leq \ell \leq t$, and obtain $\lambda^\prime_{\ell} \geq 0$ for all $1 \leq \ell \leq t$ 
and $\sum_{\ell = 1}^{t} \lambda^\prime_{\ell} = 1$
due to~\eqref{sum_of_coef}. 
Furthermore, 
\begin{align*}
X_{I} = \begin{pmatrix} X_{I \setminus \{i\}} & x_{I \setminus \{i\}} \\ x_{I \setminus \{i\}}^T & 1 \end{pmatrix}
\end{align*}
and
\begin{align*}
    s^\prime_{\ell} s ^{\prime T}_{\ell} = \begin{pmatrix} s_{\ell} s^T_{\ell}& s_{\ell} \\ s^T_{\ell} & 1 \end{pmatrix}.
\end{align*}
Thus, we can express the matrix $X_I$ as
\begin{align*}
	X_{I} = \sum_{\ell = 1}^{t} \lambda^\prime_{\ell} s^\prime_{\ell} s^{\prime T}_{\ell}.
\end{align*}
Hence, also in this case, we obtain that $X_I \in \STAB^2(G_I)$ holds, which finishes the proof.
\end{proof}

With the statement of Lemma~\ref{ESH_subgraphs} we are now ready to show that the bounds obtained by the VTESH are at least as good as the ones obtained by the ESH. 

\begin{thm}\label{LESH_at_least_as_good_as_ESH}
Let $G = (V, E)$ be a vertex-transitive graph. Then, for all $k \in \mathbb{N}_0$, 
\begin{align*}
z^\prime_k(G) \leq z_k(G).
\end{align*}
\end{thm}

\begin{proof}
Let $n$ be the number of vertices of $G$. Given that $G$ is vertex-transitive, we know that it is $r$-regular for some $r \in \mathbb{N}_0$.
If $k > n - r - 1$, then
\begin{align*}
z^\prime_k(G) = z^\prime_{n-r-1}(G) = \alpha(G) \leq z_{k}(G)
\end{align*}
holds according to the definition of the VTESH,~\eqref{ineq:relationship_zPrime} and~\eqref{ineq:relationship_z}.

Thus, what is left to show is that $z^\prime_k(G) \leq z_k(G)$ holds for all $k \in \{0, \ldots, n - r - 1\}$. From the definition of $z^\prime_k(G)$ given in~\eqref{k_th_level_LESH}, we have $z^\prime_k(G) = 1 + z_k(G^L)$, where $G^L = (V^L, E^L)$ is the local graph of $G$. Thus, we need to show that $1 + z_k(G^L) \leq z_k(G)$.

Let $(x^*,X^*)$ be an optimal solution of the SDP to compute $z_k(G^L)$, i.e., of~\eqref{theta_1} for the local graph $G^L$ with the ESC for each subgraph of order~$k$ added. 

Assume the vertex set of $G$ is $V = \{0, \ldots, n-1\}$, and the vertex set of $G^L$ is $V^L = \{0^\prime, \ldots, n-r-2^\prime\}$. Now let $\pi \colon V^L \to V$ such that $\pi(i^\prime) \in \{0, \ldots, n-1\}$ is the vertex in $G$ that corresponds to the vertex $i^\prime$ in $G^L$. 

We define $y \in \mathbb{R}^n$ as well as $Y \in \mathbb{R}^{n \times n}$ as follows. First, we set
\begin{alignat*}{3}
y_{\pi(i^\prime)} &=x^*_{i^\prime} && \quad \forall i^\prime \in V^L \\
Y_{\pi(i^\prime)\pi(j^\prime)} & = X^*_{i^\prime j^\prime} && \quad \forall i^\prime, j^\prime \in V^L.
\end{alignat*}
Furthermore, since the vertex $0$ is removed from $G$ to obtain $G^L$, we have that $0 \notin \{\pi(0^\prime), \ldots, \pi(n-r-2^\prime)\}$. Hence, we set
\begin{align*}
y_{0} = Y_{00} &= 1 \\
Y_{0j} = Y_{j0} &= x^*_j \quad \forall j \in \{\pi(0^\prime), \ldots,  \pi(n-r-2^\prime)\}.
\end{align*}
Finally, for all other indices, we define
\begin{alignat*}{3}
y_{i} &= 0 && \quad \forall i \in V \setminus \{0, \pi(0^\prime), \ldots,  \pi(n-r-2^\prime)\} \\
Y_{ij} & = 0 && \quad ~i \in V \setminus \{0, \pi(0^\prime), \ldots,  \pi(n-r-2^\prime)\} \textrm{ or } \\  
&&& \qquad \qquad  j \in V \setminus \{0, \pi(0^\prime), \ldots,  \pi(n-r-2^\prime)\}.
\end{alignat*}

This constructs an embedding of the solution $(x^*, X^*)$ into $(y, Y)$, which represents taking the solution $(x^*, X^*)$ for $G^L$ and the vertex $0$ as solution for the graph $G$. 
In particular, $y$ and $Y$ can be seen as a permutation of the 
vector $\tilde{y} = \begin{pmatrix} 1 & x^{*T} & 0^T_{r} \end{pmatrix}^T $
and the matrix
$$\tilde{Y} = \begin{pmatrix} 1 & x^{*T} & 0^T_{r} \\ x^* & X^* & 0_{(n-r-1) \times r} \\ 
0_{r} & 0_{r\times (n-r-1)} & 0_{r\times r} \end{pmatrix},$$ respectively.

Next, we show that $(y, Y)$ is a feasible solution for the SDP to compute~$z_k(G)$. From the construction of $(y, Y)$ and from 
$\diag(X^*) = x^*$ and $X^*_{ij} = 0 ~\forall \{i,j\}\in E$
it immediately follows that
$\diag(Y) = y$ and $Y_{ij} = 0$ for all $\{i,j\} \in E$.

In order to show that $\begin{pmatrix} Y & y \\ y^T & 1 \end{pmatrix} \succeq 0$ holds, it is enough to show that  $\begin{pmatrix} \tilde{Y} & \tilde{y} \\ \tilde{y}^T & 1 \end{pmatrix} \succeq 0$. Due to Schur's complement, this is the case if and only if 
$$ \tilde{Y} - \tilde{y}\tilde{y}^T
= \begin{pmatrix} 0 & 0^T_{n-r-1} & 0^T_{r} \\ 0_{n-r-1} & X^* - x^*x^{*T} & 0_{(n-r-1) \times r} \\ 
0_{r} & 0_{r\times (n-r-1)} & 0_{r\times r} \end{pmatrix} \succeq 0,
$$
which clearly holds because 
$\begin{pmatrix} X^* & x^* \\ x^{*T} & 1 \end{pmatrix} \succeq 0$ and hence with Schur's complement $X^* - x^*x^{*T} \succeq 0$.

What is left to show is that $Y_I \in \STAB^{2}(G_I)$ for all $I \subseteq V$ with $\vert I \vert = k$. First, we note that 
\begin{align}\label{property_subgraphs}
X^*_{I^L} \in \STAB^{2}(G^L_{I^L}) ~ \forall I^L \subseteq V^L: \vert I^L \vert = k.
\end{align}

Now let $I \subseteq V$ with $\vert I \vert = k$ be arbitrary but fixed, and let $\tilde{I} \subseteq I$ be the subset of vertices of $I$ that correspond to vertices of $G^L$, i.e., $\tilde{I} = I \cap \{\pi(0^\prime), \ldots, \pi(n-r-2^\prime)\}$,
and let $\tilde{k} = \vert \tilde{I} \vert$. Then, we can distinguish two cases.

Case 1: If $\tilde{k} = k$ (so $\tilde{I} = I$), then from the construction of $(y, Y)$ and~\eqref{property_subgraphs}, we obtain that $Y_I \in \STAB^{2}(G_I)$.

Case 2: If $\tilde{k} < k$,  
then we can apply the reasoning from the first case 
and obtain that $Y_{\tilde{I}} \in \STAB^{2}(G_{\tilde{I}})$. 

Now let $\{v_1, \dots, v_{k-\tilde k}\} = I \setminus \tilde{I}$.
From the construction of $(y, Y)$, we have that~$Y_{v_1j} = 0$ for all $j \in I$ if $v_1 \neq 0$, or $Y_{v_1v_1} = 1$ and $Y_{v_1j} = y_j$ for all $j \in I \setminus \{v_1\}$ if $v_1 = 0$. Hence, by applying Lemma~\ref{ESH_subgraphs}, we obtain that $Y_{\tilde{I} \cup \{v_1\}} \in \STAB^{2}(G_{\tilde{I} \cup \{v_1\} })$. 
We apply this argumentation iteratively for $v_2$, \dots, $v_{k-\tilde k}$ to finally conclude for $I = \tilde{I} \cup \{v_1, \dots, v_{k-\tilde k}\}$ that $Y_I \in \STAB^{2}(G_I)$ holds.

In both cases, we see that $Y_I$ is contained in $\STAB^{2}(G_I)$. 
Thus, $(y, Y)$ is feasible for the SDP~\eqref{theta_1} for the graph $G$ with the ESC for each subgraph of order~$k$ added, i.e., for the SDP to compute $z_k(G)$, and its objective function value equals 
\begin{align*}
\mathds{1}_n^T y = 1 + \mathds{1}_{n-r-1}^T x^* = 1 + z_k(G^L),
\end{align*}
which implies that $1 + z_k(G^L) \leq z_k(G)$, and thus $z^\prime_k(G) \leq z_k(G)$.

\end{proof}

\subsection{Behavior for Paley graphs}\label{LESH_Paley}

In Section~\ref{section_ESH_Paley}, we observed that for every $P_q$ there exists a level~$\ell(q)$ for which adding ESCs on subgraphs of orders $\{2, \ldots, \ell(q)\}$ does not yield an improvement over the Lovász theta function in terms of bounding the stability number $\alpha(P_q)$. We now investigate whether this behavior persists with the VTESH or if the VTESH provides stronger bounds for these graphs.

To this end, we first revisit the approach and results by Magsino, Mixon, and Parshall in~\cite{Magsino}. 
In that work, the authors 
obtain an upper bound on the clique number of the Paley graph $P_q$ if $q$ is a prime by exploiting the fact that $P_q$ is vertex-transitive, and hence they can assume without loss of generality that the vertex $0$ is part of a maximum clique. Therefore, the remaining vertices of such a maximum clique are in the set of neighbors of the vertex $0$. They define the graph $L_q$ as the induced subgraph of $P_q$ induced by the set of neighbors of the vertex $0$ and derive that~$\omega(P_q) = 1 + \omega(L_q)$ holds. As a consequence, the authors of~\cite{Magsino} obtain 
$\omega(P_q) \leq 1 + \vartheta(\overline{L_q})$.
By using Observation~\ref{alpha_equals_clique} that holds due to the self-complementarity of Paley graphs, we can deduce that
\begin{align*}
	b_{M}(P_q) = 1 + \vartheta(\overline{L_q})
\end{align*}
is an upper bound on $\alpha(P_q)$ whenever $q$ is a prime. Next, we investigate $L_q$ in more detail.

\begin{lem}\label{lem:GraphsIsomorphic}
	Let $P_q$ be a Paley graph such that $q$ is a prime. Then $P_q^L$ and~$\overline{L_q}$ are isomorphic.
\end{lem}
\begin{proof}
Let $Q_q$ and $NQ_q$ denote the set of nonzero quadratic residues and non-residues modulo~$q$, respectively. According to the definition of Paley graphs~$P_q$ outlined in Section~\ref{Paley_graphs}, the set of vertices of $P_q$ is $V_q = \{0\} \cup Q_q \cup NQ_q$, the set of edges of $P_q$ is $E_q = \{\{i,j\} \in V_q: i - j \in Q_q\}$ and the set of neighbors of the vertex~$0$ is $N(0) = Q_q$.

Thus, due to the definition of the local graph from Section~\ref{intro_LESH}, we have
\begin{align*}
V(P_q^L) = NQ_q \text{ and } E(P_q^L) = \left\{\{i, j\} \in NQ_q : i - j \in Q_q \right\}.
\end{align*}
Similarly, since $V(L_q) = Q_q$ and $E(L_q) = \left\{\{i, j\} \in Q_q : i - j \in Q_q \right\}$, it follows that
\begin{align*}
V(\overline{L_q}) = Q_q \text{ and } E(\overline{L_q}) = \left\{\{i, j\} \in Q_q : i - j \in NQ_q \right\}.
\end{align*}

Furthermore, $|Q_q| = |NQ_q| = \frac{q-1}{2}$ because $P_q$ is a strongly regular graph with parameters $(q, \frac{q-1}{2}, \frac{q-5}{4}, \frac{q-1}{4})$ as shown in~\cite{Godsil}.
Thus, the graphs $P_q^L$ and~$\overline{L_q}$ have the same number of vertices.
According to~\cite{Godsil}, for every $\beta \in NQ_q$ the map $\phi : V_q \to V_q$ with $\phi(i) = \beta i$ is a graph isomorphism between~$P_q$ and its complement, that maps $Q_q$ to $NQ_q$. As $P_q$ is self-complementary, it follows that $\overline{L_q}$ and $P_q^L$ are isomorphic.
\end{proof}

As a result of Lemma~\ref{lem:GraphsIsomorphic}, the bound $b_{M}(P_q)$ can be restated using $P^L_q$ instead of $\overline{L_q}$ for $q$ being a prime. Furthermore, all the cases where $q$ is not a prime can be handled with Observation~\ref{stability_local_graph} directly, and so we obtain 
\begin{align*}
	b_{M}(P_q) = 1 + \vartheta(P^L_q).
\end{align*}
as an upper bound on $\alpha(P_q)$.

In~\cite{Magsino}, the authors strengthen $b_{M}(P_q)$ as an upper bound on $\alpha(P_q)$ by considering the Schrijver refinement~\cite{Schrijver} of the Lovász theta function. In particular, for a graph~$G$ on $n$ vertices, Schrijver adds the non-negativity constraint into~\eqref{theta_2} and defines
\begin{align*}
	\vartheta^*(G) ~&=~ \max \{\langle \mathds{1}_{n \times n}, X \rangle \colon X \succeq 0, ~X \geq 0, ~ \tr(X) = 1, ~ X_{ij} = 0 ~~\forall \{i,j\} \in E\}.
\end{align*}
Galli and Letchford~\cite{GALLI2017159} argued that incorporating this non-negativity constraint $X \geq 0$ into~\eqref{theta_1} results in the same bound $\vartheta^*(G)$. Since $\vartheta^*(G) \leq \vartheta(G)$, using $\vartheta^*(G)$ provides a potentially better upper bound on $\alpha(P_q)$ than $b_{M}(P_q)$, namely
\begin{align*}
	b_{M^*}(P_q) = 1 + \vartheta^*(P^L_q).
\end{align*}
Moreover, Schrijver~\cite{Schrijver} demonstrated that the SDP formulations of both~$\vartheta(G)$ and~$\vartheta^*(G)$ can be reduced to linear programs for circulant graphs. Since $P_q$, and consequently $P^L_q$, is circulant when~$q$ is a prime, the values of both $b_M(P_q)$ and~$b_{M^*}(P_q)$ can be computed by solving a linear program in this case.

In order to assess the behavior of the VTESH for Paley graphs, we now explore the relationship between the bounds on the $k$-th level of the VTESH for Paley graphs $P_q$ and the bounds $b_{M}(P_q)$ and $b_{M^*}(P_q)$. First, we note the following.

\begin{obs}\label{ESH_level_0_1}
Let $P_q$ be a Paley graph. Then,
\begin{align*}
b_{M}(P_q) = 1 + \vartheta(P^L_q)  = z^\prime_1(P_q) = z^\prime_0(P_q).
\end{align*}
\end{obs}

Therefore, the bound $b_M(P_q)$ coincides with $z^\prime_0(P_q)$ and $z^\prime_1(P_q)$. To connect the VTESH with the bound $b_{M^*}(P_q)$, we first establish a relationship between the ESH and Schrijver's refinement of the Lovász theta function.

\begin{prop}\label{observation_Schrijver_leve_2}
For any graph $G$,
\begin{align*}
    z_2(G) \leq \vartheta^*(G) \leq z_1(G).
\end{align*}
\end{prop}

\begin{proof}
From the definition of the ESH, we know that $z_1(G) = \vartheta(G)$. Since $\vartheta^*(G) \leq \vartheta(G)$, it follows that $\vartheta^*(G) \leq z_1(G)$. Furthermore, Schrijver's refinement $\vartheta^*(G)$ is obtained by adding the non-negativity constraint $X \geq 0$ into any SDP \eqref{theta_1} or \eqref{theta_2} to compute the Lovász theta function $\vartheta(G)$. However, for subgraphs of order~$2$, the non-negativity constraint~\eqref{2_1} is only one of four inequalities describing the ESC~\eqref{ESC_2}. This implies that the bound on the second level of the ESH is at least as good as~$\vartheta^*(G)$. Therefore, $z_2(G) \leq \vartheta^*(G)$.
\end{proof}

Hence, Schrijver's refinement $\vartheta^*(G)$ lies between the first and second levels of the ESH. In the VTESH framework, Proposition~\ref{observation_Schrijver_leve_2} leads to the following relationship.

\begin{cor}\label{cor_1}
For any vertex-transitive graph $G$,
\begin{align*}
z^\prime_2(G) \leq 1 + \vartheta^*(G^L) \leq z^\prime_1(G).
\end{align*}
\end{cor}

Finally, based on Observation~\ref{ESH_level_0_1} and Corollary~\ref{cor_1}, we are now able to establish the relationship between the bound $b_M^*(P_q)$ and the VTESH for Paley graphs. 

\begin{cor}\label{cor_b_M^*}
Let $P_q$ be a Paley graph. Then,
\begin{align*}
z^\prime_2(P_q) \leq b_{M^*}(P_q) \leq b_{M}(P_q) = z^\prime_1(P_q).
\end{align*}
\end{cor}

Thus, the bound on the second level of the VTESH is at least as good as the bound $b_{M^*}(P_q)$. Naturally, we ask whether there exist Paley graphs $P_q$ for which the strict inequalities $z^\prime_2(P_q) < b_{M^*}(P_q)$ or $b_{M^*}(P_q) < b_{M}(P_q)$ hold. In either case, this would imply $z^\prime_2(P_q) < z^\prime_1(P_q)$, demonstrating that the VTESH provides better performance than the ESH for some Paley graphs, yielding tighter bounds on their stability numbers already at the second level of the hierarchy.

To answer this question, we refer to the computational results presented in~\cite{Magsino}, where the bounds $b_M(P_q)$ and $b_{M^*}(P_q)$ were computed and compared with $b_H(P_q)$, a closed-form upper bound introduced in Section~\ref{Paley_graphs}, for~$P_q$ with~$q < 3000$. According to Table 1 in~\cite{Magsino}, for 63 Paley graphs~$P_q$ with~$q~<~3000$, the inequality $b_{M^*}(P_q) < b_{M}(P_q)$ holds. Together with Corollary~\ref{cor_b_M^*}, we draw the following conclusion.

\begin{obs}\label{LESH_makes_improvement}
There exist Paley graphs $P_q$ for which
\begin{align*}
z^\prime_2(P_q) < z^\prime_1(P_q).
\end{align*}
\end{obs}

This finding is noteworthy because, as stated in Lemma~\ref{level_2}, adding ESCs for subgraphs of order~$2$ in the ESH does not improve the Lovász theta function as bound on the stability number for all Paley graphs $P_q$. Due to Observation~\ref{LESH_makes_improvement} the VTESH does bring an improvement even on level~$2$ for some Paley graphs, so this hierarchy is significantly stronger.

\subsection{Computational study}\label{ESH_comp_local_Paley}

We now computationally employ the VTESH to compute upper bounds on the stability numbers of Paley graphs $P_q$. 
For this purpose, we revisit Paley graphs $P_q$ with $q < 200$ and compute upper bounds on $z^\prime_k(P_q)$ for $k \in \{2, \dots, 10\}$.
The main goal of this computational study is to compare the bounds on $z^\prime_k(P_q)$ and on $z_k(P_q)$, as well as to investigate the quality of the bounds on stability numbers of Paley graphs obtained by utilizing the VTESH by comparing $z^\prime_k(P_q)$ to $b_{H}(P_q)$, $b_{M}(P_q)$ and $b_{M^*}(P_q)$.
We use the same computational setup and code as in Section~\ref{comp_justification}.
We again
do not consider any $q$ that is a square because of Observation~\ref{theta_exact_q_square}.
Furthermore, as we only compute upper bounds on $z^\prime_k(P_q)$ and because of the same arguments as in Section~\ref{comp_justification}, it is again possible that for some values of $q$ the upper bounds on $z^\prime_k(P_q)$ increase for higher values of~$k$, even though the actual values of~$z^\prime_k(P_q)$ decrease.

The results of this computational experiment are presented in Table~\ref{Table_3}. We provide the values of $q$, $\alpha(P_q)$ and $\vartheta(P_q)$. Then, the bounds $b_{H}(P_q)$, $b_{M}(P_q)$, and~$b_{M^*}(P_q)$ are given. Finally, we present computed upper bounds on $z^\prime_k(P_q)$ for $k \in \{2, \dots, 10\}$. If $z^\prime_k(P_q)$ for some $k \in \{2, \dots, 10\}$ yields a better integer bound on $\alpha(P_q)$ than both bounds $b_{H}(P_q)$ and $b_{M^*}(P_q)$, the respective value is printed in bold.
Whenever we can deduce that the computed upper bound on $z^\prime_k(P_q)$ is larger than $z^\prime_k(P_q)$ (because for some $k'<k$ the computed upper bound on $z^\prime_{k'}(P_q)$ is lower than the computed upper bound on $z^\prime_{k}(P_q)$), the corresponding upper bound on $z_k(P_q)$ is printed in italic.
Furthermore, we recall that the bound $b_{H}(P_q)$ is valid only for $P_q$ where $q$ is a prime, and therefore we do not present the bound $b_{H}(P_q)$ for $q = 125$. 

\begin{table}[hbt!]
	\caption{Comparison of upper bounds on 
 $\alpha(P_q)$ 
 }
	\centering
    \begin{adjustbox}{max width=\textwidth}
	\begin{tabular}{ r r r r r r | r r r r r r r r r}
		$q$ & $\alpha(P_q)$ & $\vartheta(P_q)$ & $b_{H}(P_q)$ & $b_{M}(P_q)$ & $b_{M^*}(P_q)$ & $z^\prime_2(P_q)$ & $z^\prime_3(P_q)$ & $z^\prime_4(P_q)$ & $z^\prime_5(P_q)$ & $z^\prime_6(P_q)$ & $z^\prime_7(P_q)$ & $z^\prime_8(P_q)$ & $z^\prime_9(P_q)$ & $z^\prime_{10}(P_q)$ \\
		\hline
        5 & 2 & 2.2361 & 2.0000 & 2.0000 & 2.0000 & 2.0000 & 2.0000 & 2.0000 & 2.0000 & - & - & - & - & - \\
        13 & 3 & 3.6056 & 3.0000 & 3.0000 & 3.0000 & 3.0000 & 3.0000 & 3.0000 & 3.0000 & 3.0000 & 3.0000 & 3.0000 & 3.0000 & 3.0000 \\
        17 & 3 & 4.1231 & 3.3723 & 3.3431 & 3.3431 & 3.3431 & 3.2928 & 3.0000 & 3.0000 & 3.0000 & 3.0000 & 3.0000 & 3.0000 & 3.0000 \\
        29 & 4 & 5.3852 & 4.2749 & 4.3177 & 4.3177 & 4.3177 & 4.3177 & 4.0000 & 4.0000 & 4.0000 & 4.0000 & 4.0000 & 4.0000 & 4.0000 \\
        37 & 4 & 6.0828 & 4.7720 & 4.7599 & 4.7599 & 4.7599 & 4.7599 & 4.5264 & 4.0000 & 4.0000 & 4.0000 & 4.0000 & \textit{4.0014} & \textit{4.0188} \\
        41 & 5 & 6.4031 & 5.0000 & 5.4721 & 5.4721 & 5.4721 & 5.3818 & 5.0000 & 5.0000 & 5.0000 & 5.0000 & 5.0000 & 5.0000 & 5.0000 \\
        53 & 5 & 7.2801 & 5.6235 & 5.6783 & 5.6783 & 5.6783 & 5.6761 & 5.5983 & 5.0634 & 5.0065 & 5.0004 & 5.0000 & \textit{5.0116} & \textit{5.0653} \\
		61 & 5 & 7.8102 & 6.0000 & 5.9009 & 5.8886 & 5.8886 & 5.8886 & 5.8474 & 5.5173 & 5.4913 & 5.4384 & 5.1922 & \textit{5.2066} & \textit{5.2101} \\
		73 & 5 & 8.5540 & 6.5208 & 6.3772 & 6.3772 & 6.3772 & 6.3772 & 6.3373 & 6.1044 &  \textbf{5.9488} & \textbf{5.6720} & \textbf{5.6263} & \textbf{5.6076} & \textbf{\textit{5.6081}} \\
		89 & 5 & 9.4340 & 7.1521 & 7.1553 & 7.0600 & 7.0599 & 7.0598 & 7.0403 & \textbf{6.8056} & \textbf{6.7391} &  \textbf{6.2342} &  \textbf{6.2320} & \textbf{6.2179} & \textbf{6.1961} \\
		97 & 6 & 9.8489 & 7.4462 & 7.9483 & 7.9451 & 7.9451 & 7.9449 & 7.8032 & 7.4022 & \textit{7.4942} & 7.0127 & 7.0085 & \textit{7.0116} & \textit{7.0136} \\
		101 & 5 & 10.0499 & 7.5887 & 7.2903 & 7.2891 & 7.2891 & 7.2891 & 7.2891 & 7.1738 & 7.0916 &  \textbf{6.7396} & \textbf{6.7373} & \textbf{6.6830} & \textbf{6.6628} \\
		109 & 6 & 10.4403 & 7.8655 & 8.0070 & 8.0018 & 8.0018 & 8.0017 & 7.9916 & 7.6901 & \textit{7.8152} & 7.3638 & 7.2925 & \textit{7.3116} & \textit{7.3178} \\
		113 & 7 & 10.6301 & 8.0000 & 8.3305 & 8.3305 & 8.3305 & 8.3305 & 8.3183 & 8.0452 & 8.0425 &  \textbf{7.4031} & \textbf{7.3836} &  \textbf{\textit{7.4326}} & \textbf{\textit{7.4408}} \\
        125 & 7 & 11.1803 & - & 8.5700 & 8.5700 & 8.5700 & 8.5700 & 8.4404 & 8.3771 & \textit{8.3924} & \textbf{7.2679} & \textbf{\textit{7.4508}} & \textbf{\textit{7.8357}} & \textit{8.0555} \\
        137 & 7 & 11.7047 & {8.7614} & 8.8290 & 8.8261 & 8.8261 & 8.8259 & 8.8145 & 8.7179 & \textit{8.7227} & \textbf{7.9157} & \textbf{7.9101} & \textit{8.0265} & \textit{8.0564} \\
        149 & 7 & 12.2066 & 9.1168 & 9.1884 & 9.1884 & 9.1884 & 9.1884 & 9.1884 & 9.1380 & 9.0119 &  \textbf{8.1717} &  \textbf{\textit{8.1820}} &\textbf{\textit{8.2875}} & \textbf{\textit{8.3928}} \\
        157 & 7 & 12.5300 & 9.3459 & 9.6949 & 9.6704 & 9.6699 & 9.6692 & 9.6692 & 9.4742 & \textit{9.4869} &  \textbf{8.8095} &   \textbf{8.7167} &  \textbf{\textit{8.8073}} & \textbf{\textit{8.9108}} \\
        173 & 8 & 13.1529 & 9.7871 & 10.3165 & 10.2339 & 10.2336 & 10.2331 & 10.2294 & 10.0904 & \textit{10.1463} & 9.3973 & 9.3308 & \textit{9.4518} & \textit{9.5914} \\
        181 & 7 & 13.4563 & 10.0000 & 10.3241 & 10.3207 & 10.3205 & 10.3203 & 10.3203 & 10.2424 & 10.1963 & \textbf{9.4891} &  \textbf{9.3425} & \textbf{\textit{9.5454}} & \textbf{\textit{9.6373}} \\
        193 & 7 & 13.8924 & 10.3107 & 10.5058 & 10.4379 & 10.4370 & 10.4360 & 10.4359 & 10.3395 & 10.2916 & \textbf{9.7385} & \textbf{9.5777} & \textbf{\textit{9.7065}} & \textbf{\textit{9.8309}} \\
        197 & 8 & 14.0357 & 10.4121 & 10.6517 & 10.6517 & 10.6517 & 10.6517 & 10.6500 & 10.5206 & \textit{10.5547} & \textbf{9.8377} & \textbf{9.7119} & \textbf{\textit{9.8523}} & \textit{10.0225} \\        
		\hline
	\end{tabular}
    \end{adjustbox}
	\label{Table_3} 
\end{table}

First, we examine whether applying the VTESH on Paley graphs yields better bounds on $\alpha(P_q)$ than $b_{H}(P_q)$, $b_{M}(P_q)$ and~$b_{M^*}(P_q)$, and if so, whether we also obtain better integer bounds on $\alpha(P_q)$ for the considered instances.
We note in Table~\ref{Table_3} that for $q \in \{5, 13\}$, the bounds $b_{H}(P_q)$, $b_{M}(P_q)$ and~$b_{M^*}(P_q)$ coincide with $\alpha(P_q)$, as does the bound on $z^\prime_2(P_q)$. Thus, the bounds~$z^\prime_2(P_q)$, $b_{H}(P_q)$, $b_{M}(P_q)$ and~$b_{M^*}(P_q)$ are of the same quality.  

Nevertheless, for all other considered instances, we observe that the best computed bounds on $z^\prime_k(P_q)$ are significantly better than the bounds $b_{H}(P_q)$, $b_{M}(P_q)$ and~$b_{M^*}(P_q)$. In particular, for $q \in \{17, 29, 37, 41, 53\}$ those bounds are best possible, as they coincide with $\alpha(P_q)$. Furthermore, for $11$ of $15$ considered graphs $P_q$ with $61\leq q < 200$, the best computed bounds on~$z^\prime_k(P_q)$ are better than the bounds~$b_{H}(P_q)$ and $b_{M^*}(P_q)$ by one integer. 
Thus, we see in our computational study that, in general, our bounds obtained with the VTESH are much better than the bounds proposed in~\cite{Hanson} and~\cite{Magsino}.

Second, we analyze the bounds on $z^\prime_2(P_q)$ and $b_{M^*}(P_q)$. 
Note that in our computational results for all values of $q$ we have that $z^\prime_2(P_q) \leq b_{M^*}(P_q)$, which is in accordance with the theory from Corollary~\ref{cor_b_M^*}.
We observe that for $5$ of $22$ considered instances, i.e.\ for $q \in \{89, 157, 173, 181, 193\}$, the strict inequality $z^\prime_2(P_q) < b_{M^*}(P_q)$ holds. 
Hence, empirical data show that there exist Paley graphs for which adding not only non-negativity constraints as it is done in $b_{M^*}(P_q)$, but also the inequalities \eqref{2_2}-\eqref{2_4} into the bound as it is done in $z^\prime_2(P_q)$ improves the upper bound on~$\alpha(P_q)$. Nevertheless, we note that this improvement was rather modest, and the best improvement of $0.0009$ was obtained for $P_q$ with $q = 193$. 

Finally, and most importantly, we compare the bounds on~$z_k(P_q)$ in Table~\ref{Table_2} and~$z^\prime_k(P_q)$ in Table~\ref{Table_3} and analyze whether the VTESH yields better (integer) bounds on $\alpha(P_q)$ than the ESH. Here we note that for all considered~$P_q$ with $q \geq 17$ and for all $k \in \{2, \ldots, 10\}$ the relation
$$ \lfloor z^\prime_k(P_q) \rfloor \leq \lfloor z_k(P_q) \rfloor - 1$$
holds, so on all levels $k \geq 2$, the bounds from the VTESH are at least one integer better than the bounds based on the ESH.

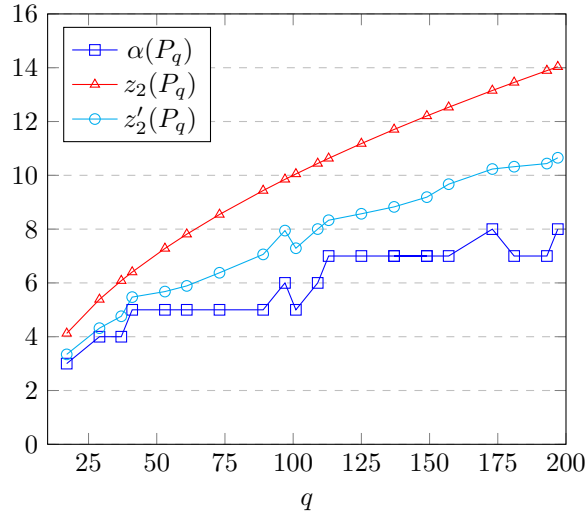
\begin{figure}[hbt!]
	\centering 
 \caption{The upper bounds on $z_2(P_q)$ and $z^\prime_2(P_q)$ for Paley graphs $P_q$ with $17 \leq q < 200$} 
\begin{tikzpicture}[scale=1.0]
			\begin{axis}[
				xlabel={$q$},
				xmin=10, xmax=200,
				ymin=0, ymax=16,
				xtick={25, 50, 75, 100, 125, 150, 175, 200},
				ytick={0, 2,  4, 6, 8, 10, 12, 14, 16},
				legend pos=north west,
				ymajorgrids=true,
				grid style=dashed,
				]			
                \addplot[
				color=blue,
				mark=square,
				]
				coordinates {
					(17,3) (29, 4) (37, 4) (41, 5) (53, 5) (61, 5) (73, 5) (89, 5) (97, 6) (101, 5) (109, 6) (113, 7) (125, 7) (137,7) (149, 7) (137,7) (149, 7) (157, 7) (173, 8) (181, 7) (193, 7) (197, 8)
				};
                \addlegendentry{$\alpha(P_q)$}
				\addplot[
				color=red,
				mark=triangle,
				]
				coordinates {
					(17, 4.1231) (29, 5.3852) (37, 6.0828) (41, 6.4031) (53, 7.2801) (61, 7.8102) (73, 8.5440) (89, 9.4340) (97, 9.8499) (101, 10.0499) (109, 10.4403) (113, 10.6301) (125, 11.1803) (137,11.7047) (149, 12.2066) (157, 12.5300) (173, 13.1529) (181, 13.4536) (193, 13.8924) (197, 14.0357)
				};
            \addlegendentry{$z_2(P_q)$}
            \addplot[
				color=cyan,
				mark=o,
				]
				coordinates {
					(17, 3.3431) (29, 4.3177) (37, 4.7599) (41, 5.4721) (53, 5.6783) (61, 5.8886) (73, 6.3772) (89, 7.0599) (97, 7.9451) (101, 7.2891) (109, 8.0018) (113, 8.3305) (125, 8.5700) (137, 8.8261) (149, 9.1884) (157, 9.6699) (173, 10.2336) (181, 10.3205) (193, 10.4370) (197, 10.6517)
				};
            \addlegendentry{$z^\prime_2(P_q)$}
			\end{axis}
		\end{tikzpicture}
	\label{figure_2a}
\end{figure}

In Figure~\ref{figure_2a}, we examine the bounds on $\alpha(P_q)$ from the second level of the ESH and the VTESH in more detail. We see that for $q \in \{17, 29, 41\}$, the bound $z^\prime_2(P_q)$ is better than the bound~$z_2(P_q)$ by one integer, while improvement by two integers is observed for $q \in \{37, 53, 61, 73, 89, 97, 109, 113\}$. Furthermore, for~$q \in \{101, 125, 137, 149, 157, 173, 181, 193\}$, an improvement of the bound by three integers is recorded. The best improvement in bound by four integers is observed for~$q = 197$. This shows that already on the second level, the VTESH clearly outperforms the ESH significantly, and the bound on $z^\prime_2(P_q)$ is better than the bound on~$z_2(P_q)$ by at least one integer for all considered Paley graphs.
This is not a surprise, as we know from Lemma~\ref{level_2} that $z_2(P_q) = \vartheta(P_q)$ for all $P_q$, so the bounds from the second level of the ESH do not improve on $\vartheta(P_q)$ as bound on $\alpha(P_q)$.

\begin{figure}[hbt!]
	\centering 
 \caption{The upper bounds on $z_8(P_q)$ and $z^\prime_8(P_q)$ for Paley graphs $P_q$ with $17 \leq q < 200$} 
\begin{tikzpicture}[scale=1.0]
			\begin{axis}[
				xlabel={$q$},
				xmin=10, xmax=200,
				ymin=0, ymax=16,
				xtick={25, 50, 75, 100, 125, 150, 175, 200},
				ytick={0, 2,  4, 6, 8, 10, 12, 14, 16},
				legend pos=north west,
				ymajorgrids=true,
				grid style=dashed,
				]			
                \addplot[
				color=blue,
				mark=square,
				]
				coordinates {
					(17,3) (29, 4) (37, 4) (41, 5) (53, 5) (61, 5) (73, 5) (89, 5) (97, 6) (101, 5) (109, 6) (113, 7) (125, 7) (137,7) (149, 7) (137,7) (149, 7) (157, 7) (173, 8) (181, 7) (193, 7) (197, 8)
				};
                \addlegendentry{$\alpha(P_q)$}
				\addplot[
				color=red,
				mark=triangle,
				]
				coordinates {
					(17,3.2864) (29, 4.4966) (37, 5.4935) (41, 5.9925) (53, 6.1915) (61, 6.9963) (73, 8.1910) (89, 9.4340) (97, 9.0777) (101, 10.0499) (109, 10.0652) (113, 10.3732) (125, 10.6415) (137,11.2181) (149, 11.9771) (157, 12.4315) (173, 13.1529) (181, 13.4536) (193, 13.8924) (197, 14.0357)
				};
            \addlegendentry{$z_8(P_q)$}
            \addplot[
				color=cyan,
				mark=o,
				]
				coordinates {
					(17,3) (29, 4) (37, 4) (41, 5) (53, 5) (61, 5.1922) (73, 5.6293) (89, 6.2320) (97, 7.0085) (101, 6.7373) (109, 7.2925) (113, 7.3836) (125, 7.4508) (137,7.9101) (149, 8.1820) (157, 8.7167) (173, 9.3308) (181, 9.3425) (193, 9.5777) (197, 9.7119)
				};
            \addlegendentry{$z^\prime_8(P_q)$}
			\end{axis}
		\end{tikzpicture}
	\label{figure_2}
\end{figure}
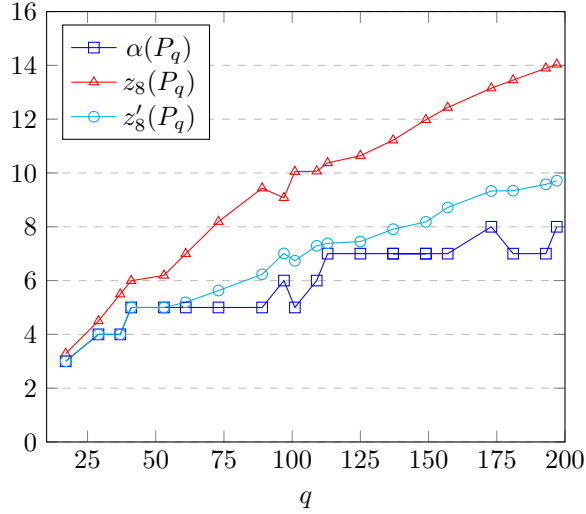  

When we consider higher levels of the VTESH, the results are even better. 
 Exemplary, we compare our computed upper bounds on $z_8(P_q)$ and~$z^\prime_8(P_q)$ in Figure~\ref{figure_2}. For $q \in \{17, 29, 37, 41, 53\}$, the bounds on $z^\prime_8(P_q)$ coincide with $\alpha(P_q)$, whereas for all other considered $P_q$, the bounds on $z^\prime_8(P_q)$ surpass the bounds on~$z_8(P_q)$ by at least one integer. Specifically, for the graph $P_q$ with $q = 61$, the bound on~$z^\prime_8(P_q)$ is better than the bound on~$z_8(P_q)$ by one integer. Furthermore, for $q = 97$, $q \in \{73, 89, 109, 113, 125, 149\}$ and $q \in \{101, 137, 157, 173, 181, 193\}$ improvements of bounds by two, three and four integers are recorded, respectively. Finally, the best improvement in bound by five integers is observed for $q = 197$. Thus, using the VTESH for computing upper bounds on $\alpha(P_q)$ produces significantly better bounds than the ESH, highlighting the strong potential of the VTESH for computing upper bounds on the stability number of Paley graphs.

However, note that our bounds obtained from the VTESH are not the best bounds on the stability number of Paley graphs known in literature. Indeed, our best bounds are weaker than the best bounds from the block-diagonal hierarchy in~\cite{Gvozdenovic2009}.
So while the VTESH substantially improves upon the ESH, it remains an open research question whether computationally trackable higher levels or further refinements of the VTESH could yield bounds as strong as those in~\cite{Gvozdenovic2009}.

\section{Conclusions and open questions}\label{Conclusion}

This work provides an in-depth analysis of ESC-based approaches for the computation of upper bounds on the stability numbers of Paley graphs. First, we examined these bounds by employing the ESH, which starts with the Lovász theta function $\vartheta$ on the first level and includes ESCs for all subgraphs of order~$k$ on level~$k$. In particular, we showed that for every Paley graph~$P_q$, there exists a level~$\ell(q) = \left\lfloor\frac{\sqrt{q}+3}{2}\right\rfloor$ for which adding ESCs on subgraphs of orders $\{2, \ldots, \ell(q)\}$ fails to provide better bounds on the stability number $\alpha(P_q)$ than the Lovász theta function. Moreover, we showed that for certain Paley graphs $P_q$ there is also no improvement of the bound on the level~$\ell(q)+1$ of the ESH.

To overcome this stagnation of the bounds, we introduced the VTESH for the stable set problem for vertex-transitive graphs, including Paley graphs. We proved that the bounds obtained by the VTESH are at least as good as the ones obtained from the ESH. A computational study for Paley graphs showed that the VTESH yields superior results compared to the ESH, often improving the bounds by several integers.

Several questions remain open. First, we established in Corollary~\ref{cor_level} that there is no improvement of the bound from the ESH up to level~$\ell(q)$ for all Paley graphs. Additionally, Theorem~\ref{thm_alpha_levelPlus1} shows no improvement on level~$\ell(q)+1$ if $\alpha(P_q) < \ell(q)$ and $q\geq 25$. However, it remains unclear whether there is an improvement of the 
bound on the ESH on level~$\ell(q)+1$ for all values of $q$ for which  
$\alpha(P_q) \geq \ell(q)$ holds, as our computations suggest for $q \leq 200$. 

For graphs with $\alpha(P_q) < \ell(q)$, it remains to be 
proven  whether no improvement occurs also on level~$\ell(q)+2$ for $q \in \{89,101,181,193\}$, and even on level~$\ell(q)+3$ for $q = 101$,  as indicated by our computations. Furthermore, it would be interesting to determine whether there is no improvement up to the level~$\ell(q) + 2$ for all $q$  satisfying  $\alpha(P_q) < \ell(q)$.

From the results for $P_q$ with $q < 200$ presented in Table~\ref{Table_2}, we observed that $\ell(q)$ typically lies in $\{\alpha(P_q) - 1, \alpha(P_q), \alpha(P_q) + 1\}$. It would be interesting to determine whether this holds for all values of~$q$, or to find a counterexample. A positive answer to this question would imply a new bound on $\alpha(P_q)$ and would thus be a remarkable result.

Another question concerns the computational implementation of the ESH and the VTESH for Paley graphs of prime order. In particular, it would be useful to investigate  whether the underlying SDP can be reduced to a linear program, as this is possible for the Lovász theta function and Schrijver's refinement, due to their circulant nature.

Future research could also perform a rigorous theoretical and computational comparison of the VTESH for Paley graphs with other hierarchical approaches like the block-diagonal SDP hierarchy introduced in~\cite{Gvozdenovic2009}, which might lead to new insights, enhanced hierarchies and tighter bounds.

Finally, it would be interesting to investigate bounds obtained by the VTESH for other classes of vertex-transitive graphs. Moreover, given that the Lovász theta function provides bounds for both the stability number and the chromatic number, it would be interesting to investigate the ESH for the graph coloring problem for Paley graphs.

\section*{Acknowledgments} We are grateful to the anonymous reviewers for their valuable comments and suggestions, which helped to improve the clarity and quality of this paper.

\section*{Funding} This research was funded in part by the Austrian Science Fund (FWF) [10.55776/DOC78]
and by the Johannes Kepler University Linz, Linz Institute of Technology (LIT):
LIT-2021-10-YOU-216.
For the purposes of open access, the authors have applied a CC BY public copyright license to all author-accepted manuscript versions resulting from this submission.

\bibliographystyle{plain}
\bibliography{Paley_bib}

\begin{thebibliography}{10}

\bibitem{AARW:15}
Elspeth Adams, Miguel Anjos, Franz Rendl, and Angelika Wiegele.
\newblock A {H}ierarchy of {S}ubgraph {P}rojection-{B}ased {S}emidefinite
  {R}elaxations for {S}ome {NP}-{H}ard {G}raph {O}ptimization {P}roblems.
\newblock {\em INFOR: Information Systems and Operational Research}, 53:40--48,
  2015.

\bibitem{Bollobas2001}
B{\'e}la Bollob{\'a}s.
\newblock {\em Random graphs}, volume~73 of {\em Cambridge Studies in Advanced
  Mathematics}.
\newblock Cambridge: Cambridge University Press, 2nd edition, 2001.

\bibitem{Broere}
Izak Broere, D.~Döman, and James Ridley.
\newblock The clique numbers and chromatic numbers of certain {P}aley graphs.
\newblock {\em Quaestiones Mathematicae}, 11:91--93, 1988.

\bibitem{Brouwer}
Andries~E. Brouwer and Hendrik Van~Maldeghem.
\newblock {\em Strongly Regular Graphs}.
\newblock Encyclopedia of Mathematics and its Applications. Cambridge
  University Press, 2022.

\bibitem{Cohen}
Stephen Cohen.
\newblock Clique numbers of {P}aley graphs.
\newblock {\em Quaestiones Mathematicae}, 11:225--231, 1988.

\bibitem{Diestel}
Reinhard Diestel.
\newblock {\em Graph {T}heory}.
\newblock Springer Publishing Company, Incorporated, 5th edition, 2017.

\bibitem{Erdos}
Paul~L. Erd{\"o}s and Alfr{\'e}d R{\'e}nyi.
\newblock Asymmetric graphs.
\newblock {\em Acta Mathematica Academiae Scientiarum Hungarica}, 14:295--315,
  1963.

\bibitem{Gaa:18}
Elisabeth Gaar.
\newblock Efficient {I}mplementation of {SDP} {R}elaxations for the {S}table
  {S}et {P}roblem.
\newblock {\em Ph.D. thesis, AlpenAdria-Universität Klagenfurt}, 2018.

\bibitem{GaarVersionsESH}
Elisabeth Gaar.
\newblock On different versions of the exact subgraph hierarchy for the stable
  set problem.
\newblock {\em Discrete Applied Mathematics}, 356:52--70, 2024.

\bibitem{Gaa:20}
Elisabeth Gaar and Franz Rendl.
\newblock A computational study of exact subgraph based {SDP} bounds for
  {M}ax-{C}ut, stable set and coloring.
\newblock {\em Mathematical Programming}, 183, 2020.

\bibitem{GaarSiebenhoferWiegeleStabBaB}
Elisabeth Gaar, Melanie Siebenhofer, and Angelika Wiegele.
\newblock An {SDP}-{B}ased {A}pproach for {C}omputing the {S}tability {N}umber
  of a {G}raph.
\newblock {\em Mathematical Methods of Operations Research}, 95(1):141--161,
  2022.

\bibitem{GALLI2017159}
Laura Galli and Adam~N. Letchford.
\newblock On the {L}ovász theta function and some variants.
\newblock {\em Discrete Optimization}, 25:159--174, 2017.

\bibitem{Godsil}
Chris Godsil and Gordon Royle.
\newblock {\em Algebraic {G}raph {T}heory}.
\newblock Algebraic Graph Theory. Springer, 2001.

\bibitem{Gru:03}
Gerald Gruber and Franz Rendl.
\newblock Computational {E}xperience with {S}table {S}et {R}elaxations.
\newblock {\em SIAM Journal on Optimization}, 13(4):1014--1028, 2003.

\bibitem{Gvozdenovic2009}
Neboj{\v{s}}a Gvozdenovi{\'c}, Monique Laurent, and Frank Vallentin.
\newblock Block-diagonal semidefinite programming hierarchies for 0/1
  programming.
\newblock {\em Operations Research Letters}, 37(1):27--31, 2009.

\bibitem{Hanson}
Brandon Hanson and Giorgis Petridis.
\newblock Refined estimates concerning sumsets contained in the roots of unity.
\newblock {\em Proceedings of the London Mathematical Society}, 122, 2019.

\bibitem{Hoffman}
Alan~J. Hoffman.
\newblock On eigenvalues and colorings of graphs.
\newblock {\em Graph Theory and its Applications}, 1970.

\bibitem{Jones}
Gareth~A. Jones.
\newblock Paley and the {Paley} graphs.
\newblock In {\em Isomorphisms, symmetry and computations in algebraic graph
  theory. Selected papers based on the presentations at the workshop on
  algebraic graph theory, Pilsen, Czech Republic, October 3--7, 2016}.
  Springer, 155--183, 2020.

\bibitem{Karp72}
Richard Karp.
\newblock Reducibility {A}mong {C}ombinatorial {P}roblems.
\newblock {\em Complexity of Computer Computations}, 40:85--103, 1972.

\bibitem{Lasserre}
Jean~B. Lasserre.
\newblock An {E}xplicit {E}xact {SDP} r{e}laxation for {N}onlinear 0-1
  {P}rograms.
\newblock In {\em Integer Programming and Combinatorial Optimization},
  293--303, 2001.

\bibitem{Laurent}
Monique Laurent.
\newblock A {C}omparison of the {S}herali-{A}dams, {L}ovász-{S}chrijver and
  {L}asserre {R}elaxations for 0-1 {P}rogramming.
\newblock {\em Mathematics of Operations Research}, 28:470--496, 2002.

\bibitem{LovSch}
L$\acute{{\rm a}}$szlo Lov$\acute{{\rm a}}$sz and Alexander Schrijver.
\newblock Cones of {M}atrices and {S}et-{F}unctions and 0–1 {O}ptimization.
\newblock {\em SIAM Journal on Optimization}, 1(2):166--190, 1991.

\bibitem{Lov:79}
L{\'a}szl{\'o} Lov{\'a}sz.
\newblock On the {Shannon} capacity of a graph.
\newblock {\em IEEE Transactions on Information Theory}, 25:1--7, 1979.

\bibitem{Magsino}
Mark Magsino, Dustin~G. Mixon, and Hans Parshall.
\newblock Linear programming bounds for cliques in {P}aley graphs.
\newblock In {\em Wavelets and Sparsity XVIII}, volume 11138, 111381H, 2019.

\bibitem{Maistrelli}
Eleni Maistrelli and David Penman.
\newblock Some colouring problems for {P}aley graphs.
\newblock {\em Discrete Mathematics}, 306(1):99--106, 2006.

\bibitem{mosek}
{MOSEK ApS}.
\newblock {\em The MOSEK optimization toolbox for MATLAB manual. Version
  10.0.}, 2024.

\bibitem{Paley}
Raymond~E. Paley.
\newblock On {O}rthogonal {M}atrices.
\newblock {\em Journal of Mathematics and Physics}, 12:311--320, 1933.

\bibitem{Sachs}
Horst Sachs.
\newblock Über selbstkomplementäre {G}raphen.
\newblock {\em Publicationes Mathematicae Debrecen}, 9:270--288, 1962.

\bibitem{Schrijver}
Alexander Schrijver.
\newblock A comparison of the {D}elsarte and {L}ov{\'a}sz bounds.
\newblock {\em IEEE Transactions on Information Theory}, 25(4):425--429, 1979.

\bibitem{SheraliAdamas}
Hanif Sherali and Warren Adams.
\newblock A {H}ierarchy of {R}elaxations {B}etween the {C}ontinuous and
  {C}onvex {H}ull {R}epresentations for {Z}ero-{O}ne {P}rogramming {P}roblems.
\newblock {\em SIAM Journal on Discrete Mathematics}, 3:411--430, 1990.

\bibitem{VanBoyd}
Lieven Vandenberghe and Stephen Boyd.
\newblock Semidefinite {P}rogramming.
\newblock {\em SIAM Review}, 38(1):49--95, 1996.

\bibitem{Xu}
Xiaodong Xu, Meilian Liang, and Haipeng Luo.
\newblock {\em Ramsey {T}heory: {U}nsolved {P}roblems and {R}esults}.
\newblock De Gruyter, Berlin, Boston, 2018.

\end{thebibliography}

\end{document}